\documentclass{amsart}
\usepackage{amsmath}
\usepackage{amsfonts}
\usepackage{amssymb}
\usepackage{amsthm}
\usepackage{mathtools}
\usepackage{url}
\usepackage{bbold}
\usepackage{appendix}

\usepackage{hyperref}
 \hypersetup{
     colorlinks=true,
     linkcolor=red,
     filecolor=blue,
     citecolor = blue,      
     urlcolor=cyan,
     }

\newcommand{\EE}{\mathbb{E}}
\newcommand{\RR}{\mathbb{R}}
\newcommand{\CC}{\mathbb{C}}
\newcommand{\NN}{\mathbb{N}}
\newcommand{\ZZ}{\mathbb{Z}}

\newcommand{\FF}{\mathbb{F}}
\newcommand{\1}{1}

\newcommand{\C}{\mathcal{C}}

\newcommand{\SSS}{\mathcal{S}}
\newcommand{\hh}{{\underline{h}}}
\newcommand{\yy}{{\underline{y}}}
\newcommand{\ww}{{\underline{w}}}

\theoremstyle{theorem}

\newtheorem{theorem}{Theorem}

\newtheorem{lemma}{Lemma}

\theoremstyle{definition}

\title[Further bounds in the polynomial Szemer\'{e}di theorem]{Further bounds in the polynomial Szemer\'{e}di theorem over finite fields}
\author{Borys Kuca}
\address{University of Manchester}
\email{borys.kuca@manchester.ac.uk}
\date{}

\begin{document}
\maketitle

\begin{abstract}
    We provide upper bounds for the size of subsets of finite fields lacking the polynomial progression $$ x, x+y, ..., x+(m-1)y, x+y^m, ..., x+y^{m+k-1}.$$ These are the first known upper bounds in the polynomial Szemer\'{e}di theorem for the case when polynomials are neither linearly independent nor homogeneous of the same degree. We moreover improve known bounds for subsets of finite fields lacking arithmetic progressions with a difference coming from the set of $k$-th power residues, i.e. configurations of the form $$ x, x+y^k, ..., x+(m-1)y^k.$$ Both results follow from an estimate of the number of such progressions in an arbitrary subset of a finite field.
\end{abstract}


\section{Introduction}
Generalizing Szemer\'{e}di's theorem on arithmetic progressions in subsets of integers \cite{szemeredi_1975}, Bergelson and Leibman proved that each dense subset of $\ZZ$ contains a configuration of the form $x, x+P_1(y), ..., x+P_m(y)$, where $y\in\ZZ\setminus{\{0\}}$ and $P_1, ..., P_m$ are polynomials with integer coefficients and zero constant term \cite{bergelson_leibman_1996}. Their proof, based on ergodic theory, does not give explicit quantitative bounds. Although no general bounds are known so far, they exist in certain special cases, for instance for $x, x+y^k, ..., x+(m-1)y^k$ with $m\geqslant 2$ and $k>1$ \cite{prendiville_2017} or for $x, x+y, x+y^2$\cite{peluse_prendiville_2019}. In the finite field analogue of the question, when we are looking for bounds on the size of $A\subset\FF_q$ lacking $x, x+P_1(y), ..., x+P_m(y)$, bounds are known in the case of $P_1, ..., P_m$ being linearly independent \cite{peluse_2019b}.

In this paper, we give the first explicit upper bounds for the sizes of subsets of finite fields lacking certain polynomial progressions. Our main result is the following.

\begin{theorem}\label{main theorem, special case}
Let $m,k\in\NN_+$, and $p$ be a prime. Suppose that $A\subset\FF_p$ lacks the progression
\begin{align}\label{union of AP and GP}
    x,\; x+y,\; ...,\; x+(m-1)y,\; x+y^m,\; ...,\; x+y^{m+k-1}
\end{align}
with $y\neq 0$. Then
\[ |A|\ll\begin{cases}
p^{1-c},\; &m = 1,2,\\
p\dfrac{(\log \log p)^4}{\log p},\; &m = 3,\\
p{(\log p)^{-c}},\; &m = 4,\\
{p}{(\log\log p)^{-c}}, \; &m>4
\end{cases}
\]
where all constants are positive, and the implied constant depends on $k$ and $m$ while $c$ depends only on $m$. For $m>4$, one can take the exponent $c$ to equal $c = 2^{-2^{m+9}}$.
\end{theorem}

It is worth noting that the exponent $c$ appearing in Theorem \ref{main theorem, special case} for $c>4$ is the same as the exponent that appeared in Gowers' bounds in Szemer\'{e}di theorem \cite{gowers_2001}.

One can think of (\ref{union of AP and GP}) as the union of an arithmetic progression and a shifted geometric progression. The cases $m=1$ and $m=2$ are in fact identical, and the bound in this case comes from the work of Peluse \cite{peluse_2019b}. Our contribution is the $m>2$ case, for which there are no previous bounds in the literature. This is the first polynomial progression for which quantitative bounds are known where polynomials in $y$ are neither linearly independent nor homogeneous of the same degree. Theorem \ref{main theorem, special case} is a special case of a more general result, which generalizes \cite{peluse_2019b} and uses it as a base case for induction.

\begin{theorem}\label{main theorem, general case}
Let $m,k\in\NN_+$, $m\geqslant 3$, and $P_m$, ..., $P_{m+k-1}$ be polynomials in $\ZZ[y]$ such that $$a_m P_m + ... + a_{m+k-1}P_{m+k-1}$$ has degree at least $m$ unless $a_m = ... = a_{m+k-1}=0$ (in particular, $P_m$, ..., $P_{m+k-1}$ are linearly independent and each of them has degree at least $m$). Let $r_m(p)$ be the size of the largest subset of $\FF_p$ lacking $m$-term arithmetic progressions and ${s_m: [p_0,\infty)\to (0,1]}$ be a decreasing function satisfying ${r_m(p)\leqslant p\cdot s_m(p)}$ for all primes $p\geqslant p_0>0$, with $s_m(n)\to 0$ as $n\to\infty$. If $A\subset \FF_p$ lacks
\begin{align}\label{generalized union of AP and GP}
    x, x+y, ..., x+(m-1)y, x+P_m(y), ..., x+P_{m+k-1}(y)
\end{align}
with $y\neq 0$, then
\begin{align*}
    |A|\ll p\cdot s_m(c p^{c})
\end{align*}
where the constants $C$, $c$, and the implied constant depend on $m, k$, and $P_m$, ..., $P_{m+k-1}$ but not on the choice of $s_m$.
\end{theorem}
The best bounds for $r_m$ currently in the literature are of the form
\[
r_m(p)\ll \begin{cases}
p\dfrac{(\log \log p)^4}{\log p},\; &m = 3 \;\text{\cite{bloom_2016}},\\
p(\log p)^{-c},\; &m = 4\; \text{\cite{green_tao_2017}},\\
p(\log\log p)^{-c},\; &m > 4\; \text{\cite{gowers_2001}}
\end{cases}
\]
yielding the bounds given in Theorem \ref{main theorem, special case}. The content of Theorem \ref{main theorem, general case} is that up to the values of constants, our bounds are of the same shape as the bounds in Szemer\'{e}di theorem. One cannot hope to do better, as each set containing (\ref{generalized union of AP and GP}) necessarily contains an $m$-term arithmetic progression. The function $s_m$ plays only an auxiliary role, allowing us to conveniently express known bounds in Szemer\'{e}di's theorem as functions defined over positive real numbers.

We prove Theorem \ref{main theorem, general case} by first proving an estimate for how many polynomial progressions a set $A\subset\FF_p$ has. This counting result is the heart of this paper; once it is proved, deducing Theorem \ref{main theorem, general case} is straightforward.

\begin{theorem}[Counting theorem]\label{counting theorem for the more difficult configuration}
Let $m\in\NN_+$ and $P_m$, ..., $P_{m+k-1}$ be polynomials in $\ZZ[y]$ such that $$a_m P_m + ... + a_{m+k-1}P_{m+k-1}$$ has degree at least $m$ unless $a_m = ... = a_{m+k-1}=0$ (in particular, $P_m$, ..., $P_{m+k-1}$ are linearly independent and each of them has degree at least $m$). Suppose that $f_0, ..., f_{m+k-1}:\FF_p\to\CC$ satisfy $|f_j(x)|\leqslant 1$ for each $0\leqslant j\leqslant m+k-1$ and $x\in\FF_p$. Then
\begin{align}\label{equality in main counting theorem}
    &\EE_{x,y\in\FF_p}\prod_{j=0}^{m-1}f_j(x+jy)\prod_{j=m}^{m+k-1}f_j(x+P_j(y))\\
    \nonumber
    =\; &\EE_{x,y\in\FF_p}\prod_{j=0}^{m-1}f_j(x+jy)\left(\prod_{j=m}^{m+k-1}\EE f_j\right) + O(p^{-c})
\end{align}
where all the constants are positive and depend on $m, k$ and polynomials $P_m$, ..., $P_{m+k-1}$ but not on $f_0, ..., f_{m+k-1}$.
\end{theorem}
Using the language of probability theory, we can interpret this result as ``discorrelation": up to an error $O(p^{-c})$, the polynomials $P_m, ..., P_{m+k-1}$ occur independently from $m$-term arithmetic progressions.

The condition imposed on the polynomials $P_m, ..., P_{m+k-1}$ may seem artificial, but Theorem \ref{counting theorem for the more difficult configuration} fails if this condition is not satisfied. As an example of failure, consider the configuration $x, x+y, x+2y, x+y^2$. Because $y^2$ has degree 2, which is less than the length of the arithmetic progression, $y^2$ is contained in the span of $x^2, (x+y)^2, (x+2y)^2$. Thus, there exist quadratic polynomials $Q_0, Q_1, Q_2$ satisfying
    \begin{align*}
        Q_0(x)+Q_1(x+y)+Q_2(x+2y)+(x+y^2) = 0.
    \end{align*}
    As a consequence, if we take $Q_3(t)=t$ and $f_j(t) = e_p(a Q_j(t))$ for $a\neq 0$, then
    \begin{align*}
        \EE_{x,y}f_0(x)f_1(x+y)f_2(x+2y)f_3(x+y^2) = 1
    \end{align*}
    while the right-hand side of (\ref{equality in main counting theorem}) in this case is $O(p^{-c})$, as $\EE f_3 = 0$. More generally, if a  linear combination of $P_m, P_{m+1}, ..., P_{k+m-1}$ has degree $d<m$, then there is a nontrivial algebraic relation connecting $x, x+y, ..., x+(m-1)y$ with some of $P_m, ..., P_{k+m-1}$, and this relation prevents discorrelation from happening. 

A natural question that one could ask at this point is whether Theorems \ref{main theorem, special case}, \ref{main theorem, general case} and \ref{counting theorem for the more difficult configuration} generalise to $\FF_q$ when $q$ is a prime power and not just a prime number. Indeed, Theorem \ref{counting theorem for the more difficult configuration} remains true if we replace $\FF_p$ by $\FF_q$, with the error $O(q^{-c})$ instead of $O(p^{-c})$. However, Theorems \ref{main theorem, special case} and \ref{main theorem, general case} no longer need to hold. In the process of going from Theorem \ref{counting theorem for the more difficult configuration} to Theorems \ref{main theorem, special case} and \ref{main theorem, general case}, one needs to apply known upper bounds for the largest subset of $\FF_q$ lacking $m$-term arithmetic progressions. These bounds differ in two extreme cases, one being $\FF_p$ and another being $\FF_q$ with $q=p^n$ and $p$ fixed. In the former case, the upper bounds for the largest subset lacking $m$-term arithmetic progressions vary from $O(\frac{p}{\log p^{1-o(1)}})$ to $O(\frac{p}{(\log \log p)^c})$ depending on the length $m$ of the arithmetic progression, as indicated earlier. For the latter, Ellenberg and Gijswijt proved a bound of the form $O(q^{1-c})$ for 3-term arithmetic progressions \cite{ellenberg_gijswijt_2016}. The fact that polynomial bounds like that cannot be attained in $\FF_p$ comes from the celebrated construction of Behrend \cite{behrend_1946}.  Therefore, the bounds that we gave in Theorems \ref{main theorem, special case} and \ref{main theorem, general case} are given for $\FF_p$ and not for all $\FF_q$. If we wanted to work in the fixed characteristic case, then the largest subset of $\FF_q$ lacking the progression (\ref{generalized union of AP and GP}) would have size $O(q^{1-c})$ for $m=3$, with constants depending on $p, k, P_3, ..., P_{2+k}$. 
In the fixed characteristic case, one can thus do strictly better than in $\FF_p$. For a general $q=p^n$, the best bound in the literature of the form $O(p^n/n)$ for subsets lacking 3-term arithmetic progression is due to Meshulam \cite{meshulam_1995}, and it implies a bound $O(p^n/n)$ for the size of subsets lacking (\ref{generalized union of AP and GP}) for $m=3$, with the constant depending on $k$ and the polynomials $P_3, ..., P_{2+k}$.

To complement these results, we prove an upper bound for the size of subsets of $\FF_p$ lacking progressions of the form
\begin{align}\label{APs with restricted common differences}
x, x+y^k, .., x+(m-1)y^k    
\end{align}
i.e.\ arithmetic progressions with $k-$th power common difference. An upper bound on subsets of $\ZZ$ lacking this configuration of the form $C\frac{N}{(\log \log N)^{c}}$, with constants depending on $m$ and $k$, was proved by Prendiville \cite{prendiville_2017} using density increment, and it naturally carries over to subsets of finite fields. Our bound works only for finite fields, where it is of the same shape as Prendiville's for $m>4$, albeit with a better exponent, and strictly improves on it for $m=3,4$.

\begin{theorem}[Sets lacking arithmetic progressions with $k$-th power differences]\label{Sets lacking arithmetic progressions with $k$-th power differences}
Suppose $A\subset \FF_p$ contains no arithmetic progression of length $m$ and common difference coming from the set of $k$-th power residues. Then
\[ |A|\ll\begin{cases}
p\dfrac{(\log \log p)^4}{\log p},\; &m = 3,\\
{p}{(\log p)^{-c}},\; &m = 4,\\
{p}{(\log\log p)^{-c}}, \; &m>4.
\end{cases}
\]
The constant $c$ depends only on $m$, and in fact for $m>4$, we can take $c = 2^{-2^{m+9}}$.
More generally,
\begin{align*}
    |A|&\ll p\cdot s_m(c'\cdot p^{c'})
\end{align*}
where $s_m$ is defined as in Theorem \ref{main theorem, general case}.  The constants $C, c'$ and the implied constants are positive and depend on $k$ and $m$.
\end{theorem}
Again, up to the values of constants involved, our bounds are optimal in the sense that they are of the same shape as the bounds in Szemer\'{e}di theorem.

We derive the bounds in Theorem \ref{Sets lacking arithmetic progressions with $k$-th power differences} using a simple argument that heavily exploits the density and equidistribution of $k$-th power residues in the finite fields. With this argument, we prove the following more general counting theorem which implies Theorem \ref{Sets lacking arithmetic progressions with $k$-th power differences}.

\begin{theorem}[Counting theorem for linear forms with restricted variables]\label{counting theorem for linear forms with restricted variables}
Let $L_1, ..., L_m$ be pairwise linearly independent linear forms in $x_1, ..., x_d$. Let $k_1,...,k_d$ be positive integers. Moreover, if $k_j>1$, assume that no linear form $L_i$ is of the form $L_i(x_1,...,x_d)=ax_j$. If $f_1,...,f_m$ satisfy $|f_i(x)|\leqslant 1$ for each $1\leqslant i\leqslant m$ and each $x\in\FF_p$, then 
\begin{align}
\label{main equation}
    \EE_{x_1,...,x_d\in\FF_p}\prod\limits_{j=1}^m f_j(L_j(x_1^{k_1},...,x_d^{k_d})) =\EE_{x_1,...,x_d\in\FF_p}\prod\limits_{j=1}^m  f_j(L_j(x_1,...,x_d))+O(p^{-c}).
\end{align}
In particular,
\begin{align*}
    &|\{(x_1,...,x_d)\in \FF_p^d: L_i(x_1^{k_1},...,x_d^{k_d})\in A\;{\rm{for}}\;1\leqslant i\leqslant m\}|\\
    &=|\{(x_1,...,x_d)\in \FF_p^d: L_i(x_1,...,x_d)\in A\;{\rm{for}}\;1\leqslant i\leqslant m\}|+O(p^{-c}).
\end{align*}
\end{theorem}

Similarly to the discussion following Theorem \ref{counting theorem for the more difficult configuration}, Theorem \ref{counting theorem for linear forms with restricted variables} remains true for $\FF_q$, but going from Theorem \ref{counting theorem for linear forms with restricted variables} to Theorem \ref{Sets lacking arithmetic progressions with $k$-th power differences} forces us to work in $\FF_p$ instead of $\FF_q$. An analogue of Theorem \ref{Sets lacking arithmetic progressions with $k$-th power differences} for $\FF_{q}$, $q=p^n$ with $p$ fixed and $n$ being the asymptotic parameter is that sets lacking $x, x+y^k, x+2y^k$ have size $O(q^{1-c})$, with the implied constant dependent on $k$. For general $q=p^n$, we would obtain that sets lacking $x, x+y^k, x+2y^k$ have size $O(p^n/n)$. 

\subsection{Known results}\label{Known results}
In this section we enumerate known bounds for subsets of $\FF_q$ or $[N]$ lacking polynomial progressions
\begin{align*}
    x,\; x+P_1(y),\; ...,\; x+P_m(y)
\end{align*}
where not all of $P_1, ..., P_m$ are linear. There are some differences between the integral and finite field settings. Most importantly, finite fields contain significantly more polynomial progressions of a given form if at least one polynomial is nonlinear. That is because a nonlinear polynomial $P$ of degree $d>1$ has only $\Theta(N^{\frac{1}{d}})$ images in $[N]$, but it is a dense subset of $\FF_q$, in the sense that there are at least $\frac{q}{d}$ images of $P$ in $\FF_q$. 



The case $m=1$ in natural numbers is often referred to as the Furstenberg-S\'{a}rk\"{o}zy theorem, and it is equivalent to finding the largest subset $A$ of natural numbers whose difference set does not intersect the values of $P$ evaluated at integers. This problem has been studied, among others, by S\'{a}rk\"{o}zy \cite{sarkozy_1978a,sarkozy_1978b}, Balog, Pelik\'{a}n, Pintz, and Szemer\'{e}di \cite{balog_pelikan_pintz_szemeredi_1994}, Slijep\u{c}evi\'{c} \cite{slijepcevic_2003}, Lucier \cite{lucier_2006}, and Rice \cite{rice_2019}. They showed that $A$ is sparse if and only if for each natural number $n$ there exists $m\in\NN$ for which $n$ divides $P(m)$, getting explicit bounds on the way; such polynomials have been called \emph{intersective}.
When $P(y)=y^k$ for $k>1$, a lower bound of the form $\Omega(N^c)$ for $0<c<1$ depending on $k$ can be obtained by trivial greedy algorithm, and the value of $c$ has been improved nontrivially by Ruzsa \cite{ruzsa_1984}. For finite fields $\FF_q$, an elementary Fourier analytic argument gives  upper bounds of the form $O(p^\frac{1}{2})$ with the implied constant depending on $k$, while the best known lower bounds are of the form $\Omega(\log p\log\log\log p)$ for infinitely many primes $p$ \cite{graham_ringrose_1990}.

In the case $m>1$, bounds have only been known in two extremes. If $P_1$, ..., $P_m$ are all homogeneous of the same degree, i.e. we have a configuration of the form
\begin{align*}
    x,\; x+c_1 y^k,\; ...,\; x+c_m y^k
\end{align*}
then Prendiville \cite{prendiville_2017} proved that all subsets of $[N]$ lacking this configuration have size $O\left(\frac{N}{(\log\log N)^{c}}\right)$ for some $c>0$ depending on $m$ and $k$. Theorem \ref{Sets lacking arithmetic progressions with $k$-th power differences} improves this result over finite fields for configurations of length 3 and 4.

The other extreme is when $P_1$, ..., $P_m$ are all linearly independent. This case has recently been tackled over finite fields by Peluse \cite{peluse_2019b} who has showed that subsets of $\FF_q$ lacking such progressions have size $O(q^{1-c})$ for $c>0$ depending on $P_1, ..., P_m$. In the case $m=2$, a specific exponent is known due to works of Bourgain and Chang \cite{bourgain_chang_2017}, Peluse \cite{peluse_2018}, and Dong, Li and Sawin \cite{dong_li_sawin_2017}. Recently, these results have been extended to the integers: Peluse and Prendiville showed that subsets of $[N]$ lacking $x, x+y, x+y^2$ have size $O(N(\log\log N)^{-c})$ \cite{peluse_prendiville_2019}, and Peluse then proved a bound of this form for subsets of $[N]$ lacking $x, x+P_1(y), ..., x+P_m(y)$ with $P_1, ..., P_m$ all having distinct degrees and zero constant terms \cite{peluse_2019a}.

\subsection{Notation, terminology, and assumptions}\label{Notation}
Throughout the paper, $p$ always denotes the characteristic and cardinality of the finite field $\FF_p$ in which we are currently working.

A function $f$ is \emph{1-bounded} if $||f||_\infty\leqslant 1$. We always assume that $f$ is a $1$-bounded function from $\FF_p$ to $\CC$ unless explicitly stated otherwise. Sometimes, we use an expression $\textbf{b}(t_1, ..., t_n)$ to denote a 1-bounded function depending only on the variables $t_1, ..., t_n$ whose exact form is irrelevant and may differ from line to line.

We denote constants by $0<c<1<C$. The exact values of these constants are generally unimportant, only their relative size, therefore we shall often use the same symbol $c, C$ to denote constants whose value changes from line to line or even in the same expression. If there are good reasons to distinguish between two constants in the same expression, we shall denote them as $c, c'$ or $C, C'$ respectively. If we need to fix a constant for the duration of an argument, we give it a numerical subscript, e.g. $c_0$. We also use asymptotic notation $f=O(g), g=\Omega(f), f\ll g$, or $g\gg f$ to denote that $|f(p)|\leqslant C |g(p)|$ for sufficiently large $p$. The constant may depend on parameters such as the length of the polynomial progression or the degrees and leading coefficients of polynomials $P_1$, ..., $P_m$ involved. However, if the asymptotic notation is used in an expression involving arbitrary functions $f_0, ..., f_m$, the constant never depends on the choice of $f_0, ..., f_m$. While it is quite common in additive combinatorics to denote the dependence of the constant on these parameters by e.g. writing $C_m$ when it depends on $m$, we refrain from doing so in order not to clutter the notation. Therefore the reader should always assume that constants depend on the shape and length of the polynomial progression, but never on the functions $f_0, ..., f_m$ weighting the progression. We shall reiterate this in the statements of our lemmas and theorems. 

We often use expected values, which we denote by $\EE_{x\in X}f(x)=\frac{1}{|X|}\sum_{x\in X}f(x)$. If the set $X$ is omitted from the notation, it is assumed that $x$ is taken from $\FF_p$ or from another specified set. 

We denote the indicator function of the set $A$ by $1_A$. The map $\mathcal{C}:x\mapsto\overline{x}$ denotes the conjugation operator. Finally, we set $e_p(x):=e(x/p)=e^{2\pi i x/p}$.

\subsection{Acknowledgements}
The author is indebted to Sean Prendiville for his unrelenting support, useful suggestions, inspiring discussions, and help with editing the paper, and to the anonymous referee for their suggestions on how to simplify the arguments in the paper.

\section{Basic concepts from additive combinatorics}\label{Gowers norms}
The purpose of this section is to describe a few basic and standard concepts that are used extensively throughout this paper. We only introduce here ideas that are essential for all the arguments. There are tools which shall only be applied in specific proofs, and these will be discussed in relevant sections.
\subsection{Fourier transform}
 Given a function $f:\FF_p\to\CC$ and $\alpha\in\FF_p$, we define its Fourier transform by the formula
\begin{align*}
    \hat{f}(\alpha):=\EE_x f(x) e_p(\alpha x).
\end{align*}
We also call $\hat{f}(\alpha)$ the \emph{Fourier coefficient of $f$ at $\alpha$}.
We define the inner product on $\FF_p$ as well as $L^s$ and $\ell^s$ norms for functions from $\FF_p$ to $\CC$ to be 
\begin{align*}
    \langle f, q\rangle :=\EE_x f(x)\overline{g(x)},\quad
    ||f||_{L^s} =\left(\EE_x |f(x)|^s\right)^\frac{1}{s},\quad {\rm{and}}\quad
     ||f||_{\ell^s} =\left(\sum_x |f(x)|^s\right)^\frac{1}{s}.
\end{align*}
for $1\leqslant s<\infty$, and we set $||f||_\infty:=||f||_{L^\infty}=||f||_{\ell^\infty}=\max\{|f(x)|: x\in\FF_p\}$.


\subsection{Gowers norms}\label{section on Gowers norms}
Let $\Delta_h f(x)=f(x+h)\overline{f(x)}$ denote the \emph{multiplicative derivative} of $f$. The $U^s$ norm of $f$ is defined as
\begin{align}\label{Gowers norm}
    ||f||_{U^s}:=\left(\EE_{x,h_1,...,h_s}\prod_{\ww\in\{0,1\}^s}\mathcal{C}^{|w|}f(x+\ww\cdot \hh) \right)^{\frac{1}{2^s}}
\end{align}
where $|w|=w_1+...+w_s$. If $f=1_A$, then $||1_A||_{U^s}^{2^s}$ is the normalized count of $s$-dimensional parallelepipeds in $A$, i.e. configurations of the form $$(x+w_1 h_1 + ... + w_s h_s)_{\ww\in \{0,1\}^s}.$$
It turns out that $||f||_{U^s}$ is a well-defined norm for $s>1$ and a seminorm for $s=1$ (for the proofs of these and other facts on Gowers norms described in this section, including Lemma \ref{von Neumann}, consult \cite{green_2007} or \cite{tao_2012}). In fact, $||f||_{U^1}=|\EE_x f(x)|=|\hat{f}(0)|$. Gowers norms enjoy several important properties that are used extensively in this paper. First, they are monotone:
\begin{align*}
    ||f||_{U^1}\leqslant ||f||_{U^2}\leqslant ||f||_{U^3}\leqslant ...
\end{align*}
Second, one can express a $U^s$ norm of $f$ in terms of a lower-degree Gowers norm of its multiplicative derivatives:
\begin{align*}
    ||f||_{U^s}^{2^s}=\EE_{h_1, ..., h_{s-k}}||\Delta_{h_1,..., h_{s-k}}f||_{U^k}^{2^k}.
\end{align*}
In particular, taking $k=2$ gives:
\begin{align*}
    ||f||_{U^s}^{2^s}=\EE_{h_1, ..., h_{s-2}}||\Delta_{h_1,..., h_{s-2}}f||_{U^2}^4.
\end{align*}
The utility of this formula for us is that $U^2$ norm is much easier to understand than the $U^s$ norms for $s>2$. In particular,
$||f||_{U^2}=||\hat{f}||_{\ell^4}$, and from the fact that $\max_{\phi\in\FF_p}|\hat{f}(\phi)|\leqslant||\hat{f}||_{\ell^4} \leqslant\max_{\phi\in\FF_p}|\hat{f}(\phi)|^{\frac{1}{2}}$ it follows that having a large $U^2$ norm is equivalent to having a large Fourier coefficient, which is the statement of \textit{$U^2$ inverse theorem}. For $s>2$, corresponding inverse theorems exist as well, but they are significantly more involved and we fortunately do not need them.

Gowers norms, introduced by Gowers in his celebrated proof of Szemer\'{e}di theorem, occur frequently in additive combinatorics because $||1_A||_{U^s}$ controls the number of $(s+1)$-term arithmetic progressions in $A$ in the following way.
\begin{lemma}[Generalized von Neumann theorem]\label{von Neumann}
Let $f_0, ..., f_s$ be $1$-bounded. Then
\begin{align*}
    |\EE_{x,y}f_0(x)f_1(x+y)...f_s(x+sy)|\leqslant\min_{0\leqslant i\leqslant s}||f_j||_{U^s}.
\end{align*}
\end{lemma}

\subsection{Counting arithmetic progressions in subsets of finite fields}\label{Counting arithmetic progressions in subsets of finite fields}
In Theorems \ref{counting theorem for the more difficult configuration} and \ref{counting theorem for linear forms with restricted variables}, we show that a certain counting operator can be expressed in terms of
\begin{align*}
    \Lambda_m(f_0, ..., f_{m-1}):=\EE_{x,y}f_0(x)f_1(x+y)...f_{m-1}(x+(m-1)y)
\end{align*}
which counts $m$-term arithmetic progressions weighted by $f_0$, ..., $f_{m-1}$. In particular, $\Lambda_m(1_A)=\Lambda_m(1_A, ..., 1_A)$ is a normalized count of $m$-term arithmetic progressions in $A$. Instead of giving the exact estimates for what this counting operator is, we want to bound it from below by an expression involving $N_m(\alpha)$, which is the smallest natural number such that $p>N_m(\alpha)$ implies that each subset of $\FF_p$ of size at least $\alpha p$ contains an $m-$term arithmetic progression. The reason why we want to have the estimate for $\Lambda_m$ in terms of $N_m$ is because the functions $N_m$ and $r'_m(n):=r_m(n)/n$ are essentially inverses, where $r_m(p)$ is the size of the largest subset of $\FF_p$ not containing an $m$-term arithmetic progression. What we mean by this is that if $r'_m$ is bounded from above by a decreasing function $s_m$, then - subject to certain conditions - $N_m$ is bounded from above by $s^{-1}_m$. The following lemma makes this precise.
\begin{lemma}\label{inverse functions}
Let $r_m(p)$ be the size of the largest subset of $\FF_p$ lacking m-term arithmetic progressions. Let $N_m(\alpha)$ be the smallest natural number such that $p > N_m(\alpha)$ implies that each subset of $\FF_p$ of size at least $\alpha p$ has an $m$-term arithmetic progression. Suppose that $s_m: [p_0,\infty) \to (0,1]$ is a decreasing function with $\lim\limits_{n\to\infty} s_m(n)=0$. Let $M_m$ be its inverse defined on $(0,\alpha_0]$, where $\alpha_0:=s_m(p_0)$. Then $r_m(p) \leqslant p s_m(p)$ for $p\geqslant p_0$ if and only if $N_m(\alpha) \leqslant M_m(\alpha)$ for $0<\alpha\leqslant\alpha_0$.
\end{lemma}
Combining Lemma \ref{inverse functions} with an averaging argument of Varnavides, we obtain the following lemma, the precise version of which has been borrowed from \cite{rimanic_wolf_2019}.
\begin{lemma}[Averaging over progressions]\label{Averaging over progressions}
Suppose $0<\alpha_0\leqslant 1$, and let $M_m:(0,\alpha_0]\to\RR_+$ be a decreasing function satisfying $N_m\leqslant M_m$. Suppose that $A\subset\FF_{p}$ has size $|A|=\alpha p$ for some $0<\alpha\leqslant\alpha_0$. Then $|\Lambda_m(1_A)|\gg 1/M_m(\alpha/2)^2$, where the implied constant depends on $m$.
\end{lemma}
We conclude this section with the proof of Lemma \ref{inverse functions}.
\begin{proof}[Proof of Lemma \ref{inverse functions}.]
Assume that $s_m$ is defined as in the statement of the lemma and that $r_m(p)\leqslant s_m(p)p$ for all prime $p\geqslant p_0$. Fix a prime number $p\geqslant p_0$ and $\alpha\in (0,\alpha_0]$. Suppose that $A\subset\FF_p$ of size $|A|=\alpha p$ lacks an $m$-term arithmetic progression. The assumption of $p\geqslant p_0$ implies that $|A|\leqslant r_m(p)\leqslant s_m(p)p$, or $\alpha\leqslant s_m(p)$. From the monotonicity of $s_m$ it follows that $p\leqslant M_m(\alpha)$.

Thus, if a subset $A\subset \FF_p$ of size $|A|=\alpha p$ for $0<\alpha\leqslant\alpha_0$ lacks $m$-term arithmetic progression, it must be that either $p\leqslant p_0$ or $p\leqslant M_m(\alpha)$, implying $N_m(\alpha)\leqslant\max\{p_0, M_m(\alpha)\}$. The definition of $p_0$ and monotonicity of $M_m$ imply that $p_0 = M_m(\alpha_0)\leqslant M_m(\alpha)$, and so $N_m(\alpha)\leqslant M_m(\alpha)$.

Conversely, suppose $N_m(\alpha)\leqslant M_m(\alpha)$ for $0<\alpha\leqslant\alpha_0$. Suppose that a set $A\subset\FF_p$ of size $|A|=\alpha p$ lacks an $m$-term arithmetic progression, and assume $0<\alpha\leqslant\alpha_0$, $p\geqslant p_0$. Then $p\leqslant N_m(\alpha)\leqslant M_m(\alpha)$, and so $\alpha \leqslant s_m(p)$. 

It thus follows that if a subset $A\subset \FF_p$ of size $|A|=\alpha p$ for $p\geqslant p_0$ lacks an $m$-term arithmetic progression, then either $\alpha\leqslant s_m(p)$ or $\alpha >\alpha_0$. If the latter holds, then $\alpha>\alpha_0$ implies $M_m(\alpha)<M_m(\alpha_0)=p_0$, and so this case is impossible whenever $p\geqslant p_0$. Thus we must have that $\alpha\leqslant s_m(p)$ whenever $p\geqslant p_0$.
\end{proof}

\section{Deriving upper bounds in Theorem \ref{main theorem, general case}}\label{deriving upper bounds}
This section is devoted to the proof of Theorem \ref{main theorem, general case} using Theorem \ref{counting theorem for the more difficult configuration} coupled with the notation from Section \ref{Counting arithmetic progressions in subsets of finite fields}.
\begin{proof}[Proof of Theorem \ref{main theorem, general case}]
Throughout this proof, all the constants are allowed to depend on $m,k$ and $P_m$, ..., $P_{m+k-1}$. From Theorem \ref{counting theorem for the more difficult configuration} it follows that
\begin{align*}
    \EE_{x,y}\prod_{j=0}^{m-1}\1_A(x+jy)\prod_{j=m}^{m+k-1}\1_A(x+P_j(y))
    =\left(\EE_{x,y}\prod_{j=0}^{m-1}\1_A(x+jy)\right)\alpha^k + O(p^{-c})
\end{align*}
If $A\subset\FF_p$ for $p\geqslant p_0$ has size $|A|=\alpha p$ and lacks progressions (\ref{generalized union of AP and GP}), then the expression on the left-hand side is $O(p^{-1})$, and so
\begin{align}\label{inequality with inverse functions}
    \left(\EE_{x,y}\prod_{j=0}^{m-1}\1_A(x+jy)\right)\alpha^k\ll p^{-c}.
\end{align}
Let $M_m$ be the inverse function for $s_m$ on $(0,\alpha_0]$, where $\alpha_0=s_m(p_0)$, and set $M=M_m(\alpha/2)$. The assumption $p\geqslant p_0$ and the fact that $s_m$ is decreasing imply that $0<\alpha\leqslant\alpha_0$. Applying Lemma \ref{Averaging over progressions} to (\ref{inequality with inverse functions}) gives $\alpha^{k}M^{-2}\ll p^{-c}$. Behrend's construction implies that $M$ grows faster than polynomially in $\alpha$: that is, for each $C>1$ there exists $c>0$ such that $M\geqslant c \alpha^{-C}$ \cite{behrend_1946}. Consequently, we have $M^{-3}\ll p^{-c}$ which implies that $M\gg p^c$ for a different constant $0<c<1$. From monotonicity of $s_m$ it follows that $\alpha\leqslant 2s_m(c p^c)$. 

\end{proof}

To illustrate the last bit of the above proof, we take Gowers's \cite{gowers_2001} estimate $$N_m(\alpha)\leqslant 2^{2^{\alpha^{-C}}}$$ for $m > 4$. Combined with $N_m(\alpha/2)\gg p^c$, it gives the inequality $2^{2^{C\alpha^{-C}}} \gg p^c$. After rearranging, it yields 
\begin{align*}
    \alpha\ll\frac{1}{(\log\log p)^c}.
\end{align*}
Note that the function $s_m(p)=(\log_2\log_2 p)^{-c}$ is precisely the inverse function of $M_m(\alpha)=2^{2^{\alpha^{-C}}}$ for an appropriate choice of constants.

\section{Proof of Theorem \ref{counting theorem for the more difficult configuration}}\label{main section}
Finally, we come to the main part of the paper, which is the proof of the counting theorem for the progression (\ref{generalized union of AP and GP}). Like before, all the constants here are allowed to depend on $m, k$ and $P_m, ..., P_{m+k-1}$. First, we lexicographically order the set $\NN^2_+$, i.e.
\begin{align*}
    (m,k)<(m',k')\iff m<m'\; {\rm{or}}\; (m=m'\; {\rm{and}}\; k<k').
\end{align*}
We induct on $(m,k)$ by following the lexicographic order on $\NN_+^2$. Let $\SSS(m,k)$ denote the statement of Theorem \ref{counting theorem for the more difficult configuration} for $(m,k)$; that is, $\SSS(m,k)$ holds iff for all linearly independent polynomials $P_m, ..., P_{m+k-1}$ of degree at least $m$ that do not span a polynomial of degree less than $m$ there exists a constant $c>0$ such that for all 1-bounded functions $f_0, ..., f_{m+k-1}$, we have
\begin{align*}
    &\EE_{x,y}\prod_{j=0}^{m-1}f_j(x+jy)\prod_{j=m}^{m+k-1}f_j(x+P_j(y))
    \\=&\left(\EE_{x,y}\prod_{j=0}^{m-1}f_j(x+jy)\right)\prod_{j=m}^{m+k-1}\EE f_j + O(p^{-c}).
\end{align*}

$\SSS(1,k)$ and $\SSS(2,k)$ follow from the work of Peluse \cite{peluse_2019b}, and they shall serve as our base cases. In the inductive step, we have to prove two cases:
\begin{enumerate}
    \item $\SSS(m,1)$, assuming the statement holds for all $(m',k')<(m,1)$ (although we shall only need to invoke $\SSS(m-1,2)$).
    \item $\SSS(m,k)$ for $k>1$, assuming it holds for $\SSS(m',k')$ with $(m',k')<(m,k)$.
\end{enumerate}
The first case turns out to be the simpler of the two, and we shall carry it out promptly. The second case is much more involved, and it is where most of the difficulties lie.

Throughout this section, we denote the counting operator appearing in the statement of the Theorem \ref{counting theorem for the more difficult configuration} by $\Lambda$ with appropriate subscripts. Thus,
\begin{align*}
    \Lambda_{m, P_m, ..., P_{m+k-1}}(f_0, ..., f_{m+k-1}) := \EE_{x,y}\prod_{j=0}^{m-1}f_j(x+jy)\prod_{j=m}^{m+k-1}f_j(x+P_j(y)).
\end{align*}
In particular, $\Lambda_m$ denotes the counting operator for $m$-term arithmetic progressions:
\begin{align*}
    \Lambda_m(f_0, ...,f_{m-1}):= \EE_{x,y}\prod_{j=0}^{m-1}f_j(x+jy).
\end{align*}
When $m,k$, and $P_m, ..., P_{m+k-1}$ are clear out of the context, we shall suppress the subscripts and denote the operator just by $\Lambda$.

\subsection{Proof of $\SSS(m,1)$ assuming $\SSS(m-1,2)$}
As advertised earlier, we first prove the inductive step for $\SSS(m,1)$. Let $P$ be a polynomial of degree at least $m$. Our goal is to show that the counting operator
\begin{align}\label{Lambda_{m,P}}
    \Lambda_{m, P}(f_0,...,f_m)=\EE_{x,y}\left(\prod_{j=0}^{m-1}f_j(x+jy)\right)f_m(x+P(y))
\end{align}
is in fact controlled by an operator involving an arithmetic progression of length $m-1$ of difference functions of $f_1, ..., f_{m-1}$. To accomplish this, we first rewrite (\ref{Lambda_{m,P}}) as
\begin{align*}
    \EE_{x}f_0(x)\EE_y \left(\prod_{j=1}^{m-1}f_j(x+jy)\right)f_m(x+P(y)).
\end{align*}
Applying the Cauchy-Schwarz inequality in $x$ together with 1-boundedness of $f_0$, changing variables, translating $x\mapsto x-y$, and finally using the triangle inequality, we obtain that
\begin{align*}
    |\Lambda_{m,P}(f_0,...,f_m)|^2 &\leqslant\EE_x \left|\EE_y \left(\prod_{j=1}^{m-1} f_j(x+jy)\right)f_m(x+P(y))\right|^2\\
    &\leqslant\EE_{x,y,h}\left(\prod_{j=1}^{m-1}\Delta_{jh} f_j(x+jy)\right)\overline{f_m(x+P(y))}f_m(x+P(y+h))\\
    &\leqslant\EE_h\left|\EE_{x,y}\left(\prod_{j=1}^{m-1}\Delta_{jh} f_j(x+(j-1)y)\right)\overline{f_m(x+P(y)-y)}f_m(x+P(y+h)-y)\right|.
\end{align*}
By the pigeonhole principle, there exists $h\neq 0$ such that 
\begin{align*}
    |\Lambda_{m,P}(f_0,...,f_m)|^2 &\leqslant \left|\EE_{x,y}\left(\prod_{j=1}^{m-1}\Delta_{jh} f_j(x+(j-1)y)\right)\overline{f_m(x+P(y)-y)}f_m(x+P(y+h)-y)\right|+O(p^{-1})\\
    &=|\Lambda_{m-1, P_m, P_{m+1}}(g_0, ..., g_{m-2}, \overline{f_m}, f_m)|+O(p^{-1})
\end{align*}
where we set
\begin{align*}
P_m(y) = P(y)-y, \quad  P_{m+1}(y) = P(y+h)-y \quad {\rm{and}} \quad
g_j(t) = \Delta_{(j+1)h}f_{j+1}(t).
\end{align*}
From $h\neq 0$ it follows that $P_m$, $P_{m+1}$ are linearly independent. Moreover, for any $(a,b)\neq (0,0)$, the polynomial $a P_m + b P_{m+1}$ has degree at least $m-1$, attaining this degree precisely when $a+b=0$. We have thus reduced the study of $\Lambda_{m,P}$ to the analysis of $\Lambda_{m-1,P_m,P_{m+1}}$, and so we are in the $\SSS(m-1,2)$ case. Applying Theorem \ref{counting theorem for the more difficult configuration} for this case, we see that 
\begin{align*}
    |\Lambda_{m,P}(f_0,...,f_m)|^2 &\leqslant|\Lambda_{m-1, P_m, P_{m+1}}(g_0, ..., g_{m-2}, \overline{f_m},f_m)|+O(p^{-1})\\
    &=|\Lambda_{m-1}(g_0,...,g_{m-2})|\cdot|\EE f_m|^2 + O(p^{-c})\\
    &\leqslant|\EE f_m|^2 + O(p^{-c})
\end{align*}
and hence
\begin{align*}
    |\Lambda_{m,P}(f_0,...,f_m)|\leqslant |\EE f_m|+O(p^{-c}).
\end{align*}

We have established so far that the $U^1$ norm of $f_m$ controls $\Lambda_{m,P}(f_0,...,f_m)$ up to a power-saving error, i.e. $||f_m||_{U^1}=0$ implies $|\Lambda_{m,P}(f_0,...,f_m)|=O(p^{-c})$. To utilise this fact, we decompose $f_m = \EE f_m + (f_m - \EE f_m)$ and split $\Lambda_{m,P}$ accordingly. The term involving $f_m - \EE f_m$ has size at most $O(p^{-c})$ because $\EE(f_m-\EE f_m) = 0$, and so
\begin{align*}
    \Lambda_{m, P}(f_0, ..., f_m) = \Lambda_m(f_0, ...,f_{m-1})\EE f_m + O(p^{-c}),
\end{align*}
as required.

\subsection{Proof of $\SSS(m,k)$, $k>1$}
Our next goal is to prove $\SSS(m,k)$ whenever $k>1$. The natural thing to try would be to prove this case in a similar manner we proved $\SSS(m,1)$; that is, to apply the Cauchy-Schwarz inequality to the counting operator
\begin{align}\label{counting operator}
    \Lambda_{m, P_m, ..., P_{m+k-1}}(f_0, ..., f_{m+k-1})
\end{align}
and bound it by the counting operator of
\begin{align*}
    \Lambda_{m-1,Q_m, R_m, ..., Q_{m+k-1}, R_{m+k-1}}(g_0, ..., g_{m-2}, \overline{f_m}, f_m, ..., \overline{f_{m+k-1}}, f_{m+k-1})
\end{align*}
where
\begin{align*}
Q_j(y) = P_j(y)-y, \quad 
R_j(y) = P_j(y+h)-y \quad {\rm{and}} \quad
g_j(t) = \Delta_{(j+1)h}f_{j+1}(t).
\end{align*}
However, this simple extension of the method used to prove $\SSS(m,1)$ does not work because there is no guarantee that $Q_m, R_m, ..., Q_{m+k-1}, R_{m+k-1}$ are linearly independent (and in general, they may not be), nor that any nonzero linear combination of them has degree at least $m-1$. To illustrate this problem, we look at 
\begin{align*}
    x,\; x+y,\; x+2y,\; x+y^3,\; x+y^4.
\end{align*}
Applying the Cauchy-Schwarz inequality and translating by $x\mapsto x-y$, we control this configuration by the counting operator of the configuration
\begin{align*}
    x,\; x+y,\; x+y^3-y,\; x+(y+h)^3-y,\; x+y^4-y,\; x+(y+h)^4-y.
\end{align*}
Note that the polynomials $y,\; y^3-y,\; (y+h)^3-y,\; y^4-y,\; (y+h)^4-y$ have degree at most 4, and there are 5 of them, hence there exist $a_1, ..., a_5, b$ not all zero such that
\begin{align*}
    a_1 y + a_2 (y^3-y) +a_3 ((y+h)^3-y) + a_4 (y^4-y) + a_5 ((y+h)^4-y) = b.
\end{align*}
Consequently, one cannot apply induction hypothesis to this configuration. One therefore needs to come up with a different method. 

Throughout this section, we let
\begin{align*}
    \Lambda:=\Lambda_{m, P_m, ..., P_{m+k-1}}.
\end{align*} 
Our general strategy for $\SSS(m,k)$, $k>1$ is to gradually replace each of $f_m, ..., f_{m+k-1}$ by additive characters. Our method follows very closely the techniques in \cite{peluse_2019a, peluse_2019b, peluse_prendiville_2019}, and we shall point the reader to the relevant statements in these papers for comparison. To replace arbitrary functions by characters, we introduce an inner induction loop, much like in the proof of Theorem 2.1 of \cite{peluse_2019b}. For $0\leqslant r\leqslant k$, let $\SSS(m,k,r)$ denote the statement that for all polynomials $P_m$, ..., $P_{m+k-1}$ satisfying the conditions of Theorem \ref{counting theorem for the more difficult configuration}, there exists $c>0$ such that
\begin{align*}
    &\Lambda_{m, P_m, ..., P_{m+k-1}}(f_0, ..., f_{m+r-1}, e_p(a_{m+r}\cdot), ..., e_p(a_{m+k-1}\cdot))\\
    &=\Lambda_{m}(f_0, ..., f_{m-1})\prod_{j=m}^{m+r-1}\EE f_j \prod_{j=m+r}^{m+k-1}1_{a_{j}=0}+O(p^{-c}).
\end{align*}
 for all 1-bounded functions $f_0, ..., f_{m+r-1}:\FF_p\to\CC$. We note that $\SSS(m,k,r)$ is the special case of $\SSS(m,k)$ restricted to the situation when $f_j = e_p(a_j\cdot)$ for $m+r\leqslant j\leqslant m+k-1$, and $\SSS(m,k,k)$ is equivalent to $\SSS(m,k)$. We shall therefore deduce $\SSS(m,k)$ by inducting on $0\leqslant r\leqslant k$ for fixed $(m,k)$. We start by proving the base case $\SSS(m,k,0)$, which by the homomorphism property of additive characters and assumptions on the polynomials $P_m, ..., P_{m+k-1}$ reduces to the statement in the following lemma.
\begin{lemma}\label{Vanishing of expression II}
Let $a\in\FF_p^\times$ and $m\in\NN_+$. Suppose that $P\in\ZZ[y]$ has degree at least m and that the functions $f_0, ..., f_{m-1}:\FF_p\to\CC$ are 1-bounded. Then
\begin{align}\label{vanishing of the correlation with character}
    |\Lambda_{m,P}(f_0, ..., f_{m-1}, e_p(a\cdot))|\leqslant O(p^{-c})
\end{align}
for a constant $c>0$ depending on $m$ and $P$ but not on $a$ or $f_0, ..., f_{m-1}$.
\end{lemma}
\begin{proof}
We prove the statement by induction on $m$. For $m=1$, we have
\begin{align*}
    |\EE_{x,y}f_0(x)e_p(a(x+P(y))|=|\EE_{x}f_0(x)e_p(ax)|\cdot|\EE_y e_p(aP(y))|\ll p^{-c}
\end{align*}
by the 1-boundedness of $f_0$ and Weyl differencing. 

Suppose $m>1$. Then an application of the Cauchy-Schwarz inequality to remove $f_0$ followed by a change of variables gives
\begin{align*}
    |\Lambda_{m,P}(f_0, ..., f_{m-1}, e_p(a\cdot))|\leqslant |\EE_h\Lambda_{m-1,Q_h}(e_p(-a\cdot)\Delta_h f_1, ..., \Delta_{(m-1)h}f_{m-1}, e_p(a\cdot))|,
\end{align*}
where $Q_h(y) := P(y+h)-P(y)$. For $h\neq 0$, the degree of $Q_h$ satisfies $$\deg Q_h = \deg P - 1\geqslant m - 1.$$
By inductive hypothesis and triangle inequality, (\ref{vanishing of the correlation with character}) holds for $m$. The lemma follows by induction.
\end{proof}

The heart of the proof of $\SSS(m,k)$ for $k>1$ is thus to show that $\SSS(m,k,r+1)$ can be deduced from $\SSS(m,k,r)$. The next lemma states this more formally.
\begin{lemma}\label{S(m,k,r) implies S(m,k,r+1)}
Let $m\geqslant 3$, $k\geqslant 2$ and $0\leqslant r<k$ be natural numbers. Assume $\SSS(m,k,r)$ holds. Then $\SSS(m,k,r+1)$ holds as well.
\end{lemma}
The case $\SSS(m,k)$, $k>1$ thus follows by inducting on $r$ and the observation that $\SSS(m,k,k)=\SSS(m,k)$.

From now on, assume $(m,k,r)$ is fixed. In the remainder of this section, we outline the proof of Lemma \ref{S(m,k,r) implies S(m,k,r+1)}. We formulate consecutive steps of the proof as lemmas to be proved separately in the next section. Our first task in proving Lemma \ref{S(m,k,r) implies S(m,k,r+1)} is to show that $\Lambda$ is controlled by some Gowers norm of $f_{m+r}$. This follows from the so-called PET induction scheme, which originally appeared in Bergelson and Leibman's ergodic-theoretic proof of the polynomial Szemer\'{e}di theorem \cite{bergelson_leibman_1996} and was subsequently applied in the works of Prendiville and Peluse \cite{prendiville_2017, peluse_2019b, peluse_prendiville_2019} and Tao and Ziegler \cite{tao_ziegler_2008, tao_ziegler_2016, tao_ziegler_2018}.

\begin{lemma}[PET induction, Proposition 2.2 of \cite{peluse_2019b}]\label{PET induction}
Let $P_1, ..., P_l$ be nonconstant polynomials in $\ZZ[y]$ such that $P_i - P_j$ is nonconstant whenever $i\neq j$. Then for any $1\leqslant j\leqslant l$ there exist $s\in\NN$ and $0<\beta\leqslant 1$, depending only on the degrees and leading coefficients of $P_1, ..., P_l$, such that
\begin{align*}
    |\Lambda_{x, x+P_1(y), ..., x+P_l(y)}(f_0, ..., f_l)|\leqslant ||f_j||_{U^s}^\beta+O(p^{-\beta}).
\end{align*}
for all 1-bounded functions $f_0, ..., f_l:\FF_p\to\CC$.
\end{lemma}
Our statement differs slightly from the statement of Proposition 2.2 in \cite{peluse_2019b} in that Peluse did not mention explicitly our condition that the difference between any two polynomials $P_i$, $P_j$ cannot be constant. However, she assumed throughout her paper that $P_1, ..., P_l$ were distinct polynomials with zero constant terms, which implies our condition. In our paper, the polynomials may have nonzero constant terms, in which case we replace $P_i(y)$ by $P'_i(y):=P_i(y)-P_i(0)$ and $f_i(t)$ by $f_i'(t):=f_i(t+P_i(0))$, so that $f_i(x+P_i(y))=f'_i(x+P_i'(y))$. The facts that $f_i$ and $f'_i$ have the same Gowers norms and that $P_1'$, ..., $P_l'$ are all distinct polynomials with zero constant terms allows us to reduce to the case covered in Proposition 2.2 of \cite{peluse_2019b}.

Our next step is to decompose $f_{m+r}$ into three terms using a decomposition based on the Hahn-Banach theorem.
\begin{lemma}[Hahn-Banach decomposition, Proposition 2.6 of \cite{peluse_2019b}]\label{Hahn-Banach decomposition}
Let $f:\FF_p\to\CC$ and $||\cdot||$ be a norm on the space of $\CC$-valued functions from $\FF_p$. Suppose $||f||_{L^2}\leqslant 1$. Then there exists a decomposition
\begin{align*}
    f = f_a + f_b + f_c
\end{align*}
with $||f_a||^*\leqslant p^{\delta_1}$, $||f_b||_{L^1}\leqslant p^{-\delta_2}$, $||f_c||_{L^\infty}\leqslant p^{\delta_3}$, $||f_c||\leqslant p^{-\delta_4}$ provided
\begin{align}\label{delta inequality}
    p^{\delta_4-\delta_1}+p^{\delta_2-\delta_3}\leqslant\frac{1}{2}.
\end{align}
\end{lemma}

This decomposition was pioneered by Gowers and Wolf in their work on true complexity of linear forms \cite{gowers_wolf_2011a, gowers_wolf_2011b, gowers_wolf_2011c, gowers_2010}. The variant that we are using is due to Peluse and appeared in \cite{peluse_2019b, peluse_prendiville_2019}. The dual norm in the statement of Lemma \ref{Hahn-Banach decomposition} is defined by $||f||^*=\sup\{|\langle f,g\rangle|: ||g||_\infty\leqslant 1\}$.

The notation has already become rather formidable, and it will become even more so in the further part of the proof. To make it more palatable, we let $P_j(y) := jy$ for ${0\leqslant j\leqslant m-1}$ and $f_j(t) :=e_p(a_j t)$ for $m+r+1\leqslant j\leqslant m+k-1$ for the rest of Section \ref{main section}.

Using Hahn-Banach decomposition, we can write $f_{m+r}$ as a sum of three functions: the first has not too big $U^s$-dual norm, the second has small $L^1$ norm, and the third has a small $U^s$ norm and not too big $L^\infty$ norm. By taking appropriate values of $\delta_1, \delta_2, \delta_3, \delta_4$, we get rid of two error terms and only work with $f_a$. This gives us control over $\Lambda_{m, P_m, ..., P_{m+k-1}}(f_0, ..., f_{m+k-1})$ by the $U^s$ norm of a dual function
\begin{align}\label{dual function}
    F(x):=\EE_y\prod_{j=0}^{m+r-1} f_j(x+P_j(y)-P_{m+r}(y))\prod_{j=m+r+1}^{m+k-1}f_j(x+P_j(y)-P_{m+r}(y))
\end{align}
and allows us to essentially replace $f_{m+r}$ in the $\Lambda_{m, P_m, ..., P_{m+k-1}}$ operator by a character. We call $F$ a ``dual function" because $\Lambda_{m, P_m, ..., P_{m+k-1}}(f_0,...,f_{m+k-1})=\langle F,\overline{f_{m+r}}\rangle$.

In general, higher degree Gowers norms control lower degree norms but the converse is not true. For the special case of the dual function $F$, we however show that $||F||_{U^s}$ is indeed controlled by $||F||_{U^2}$ for any $s\in\NN$. We achieve this in the lemma below which we encourage the reader to compare with Lemma 4.1 of \cite{peluse_2019b} and Proposition 6.6 of \cite{peluse_prendiville_2019}.

\begin{lemma}[Degree lowering]\label{degree lowering}
Let $F$ be defined as in (\ref{dual function}). For each $s> 2$,
\begin{align*}
    ||F||_{U^{s-1}}= \Omega(||F||_{U^s}^{2^{2s-1}})-O(p^{-c})
\end{align*}
for $c>0$ depending on $m, k,$ and $P_m, ..., P_{m+k-1}$ but not on $f_0, ..., f_{m+k-1}$.
As a consequence,
\begin{align*}
    ||F||_{U^2} = \Omega(||F||_{U^s}^{2^{(s-2)(s+2)}}) - O(p^{-c}).
\end{align*}
\end{lemma}

Having a control by the $U^2$ norm of the dual function $F$ is important because this norm is in turn controlled by the $U^{1}$ norms of the component functions $f_m, ..., f_{m+r-1}, f_{m+r+1}, ..., f_{m+k-1}$, which follows from Lemma \ref{Vanishing of expression II} coupled with $\SSS(m,k-1)$. Recalling that $f_j(t) :=e_p(a_j t)$ for $m+r+1\leqslant j\leqslant m+k-1$ and so $||f_j||_{U^1}=1_{a_j=0}$ for these values of $j$, we obtain the following lemma.

\begin{lemma}[$U^1$ control of the dual]\label{$U^1$ control of the dual}
Let $F$ be defined as in (\ref{dual function}). Then
\begin{align*}
    ||F||_{U^2}\leqslant\min_{m\leqslant j\leqslant m+r-1}||f_j||_{U^1}^\frac{1}{2}\cdot\prod_{j=m+r+1}^{m+k-1}1_{a_j=0}+O(p^{-c})
\end{align*}
for some $c>0$ depending on $m, k,$ and $P_m, ..., P_{m+k-1}$ but not on $f_0, ..., f_{m+k-1}$.
\end{lemma}

Combining the estimates of two previous lemmas with the Hahn-Banach decomposition, we get a control of the $\Lambda$ operator by $U^1$ norms of  $f_m, ..., f_{m+r-1}, f_{m+r+1}, ...,$  $f_{m+k-1}$.

\begin{lemma}[$U^1$ control of $\Lambda$, cf. Theorem 7.1 of \cite{peluse_prendiville_2019}]\label{$U^1$ control of the operator}
There exist constants $c,c'>0$ and $s\in\NN$ depending only on $m,k, P_m, ..., P_{m+k-1}$ but not on $f_0, ..., f_{m+k-1}$ such that
\begin{align*}
    |\Lambda(f_0, ..., f_{m+k-1})| &\ll p^{c'} \min_{m\leqslant j\leqslant m+r-1}||f_j||_{U^1}^{2^{-s}}\cdot\prod_{j=m+r+1}^{m+k-1}1_{a_j=0} +p^{-c}.
\end{align*}
\end{lemma}

Having established Lemma {\ref{$U^1$ control of the operator}}, it is straightforward to prove $\SSS(m,k,r+1)$; however, the argument is slightly different for $r=0$ and $r>0$. If $r=0$, then by Lemma \ref{$U^1$ control of the operator} we have 
\begin{align*}
    |\Lambda(f_0, ..., f_{m+k-1})| &\ll p^{c'} \prod_{j=m+1}^{m+k-1}1_{a_j=0} +p^{-c}.
\end{align*}
If not all of $a_{m+1}$, ..., $a_{m+k-1}$ are zero, then
\begin{align*}
    |\Lambda(f_0, ..., f_{m+k-1})| &\ll p^{-c}.
\end{align*}
Otherwise we are in the case $\SSS(m,1)$. Combining these two alternatives gives $\SSS(m,k,1)$. 

If $r>0$, we split each of $f_m$, ..., $f_{m+r-1}$ into $f_j = \EE f_j + (f_j-\EE f_j)$, and decompose $\Lambda$ accordingly. Then $\Lambda(f_0, ..., f_{m+k-1})$ splits into the main term
\begin{align*}
    \Lambda(f_0, ..., f_{m-1}, \EE f_m, ..., \EE f_{m+r-1}, f_{m+r}, ..., f_{m+k-1})
\end{align*}
and $2^r-1$ error terms, each of which involves at least one $f_j - \EE f_j$ for ${m\leqslant j\leqslant m+r-1}$. Using Lemma \ref{$U^1$ control of the operator}, each of the error terms has size $O(p^{-c})$; hence
\begin{align*}
   \Lambda(f_0, ..., f_{m+k-1}) &= \Lambda(f_0, ..., f_{m-1}, \EE f_m, ..., \EE f_{m+r-1}, f_{m+r}, ..., f_{m+k-1}) +O(p^{-c})\\
   &=\Lambda_{m, P_{m+r}, ... P_{m+k-1}}(f_0, ..., f_{m-1}, f_{m+r}, ..., f_{m+k-1})\prod_{j=m}^{m+r-1}\EE f_j +O(p^{-c}).
\end{align*}
Applying the $\SSS(m,k-r)$ case, we can split $\Lambda_{m, P_{m+r}, ... P_{m+k-1}}$
\begin{align*}
&\Lambda_{m, P_{m+r}, ... P_{m+k-1}}(f_0, ..., f_{m-1}, f_{m+r}, ..., f_{m+k-1})\\
&= \Lambda_m(f_0, ..., f_{m-1})\; \EE f_{m+r}\prod_{j=m+r+1}^{m+k-1}1_{a_j = 0}+O(p^{-c})
\end{align*}
and hence
\begin{align*}
    \Lambda(f_0, ..., f_{m+k-1}) = \Lambda_m(f_0, ..., f_{m-1})\prod_{j=m}^{m+r} \EE f_j \prod_{j=m+r+1}^{m+k-1}1_{a_j = 0} + O(p^{-c}).
\end{align*}
This proves $\SSS(m,k,r+1)$ for $r>0$.



\subsection{Proofs of Lemmas \ref{degree lowering}, \ref{$U^1$ control of the dual} and \ref{$U^1$ control of the operator}}
While in the previous section we outlined the proof of $\SSS(m,k)$ for $k>1$, here we derive the technical lemmas which are used in this proof.
\begin{proof}[Proof of Lemma \ref{degree lowering}]
This proof follows the path of Proposition 6.6 in \cite{peluse_prendiville_2019}. The main idea is to write the $U^s$ norm of the dual function $F$ as an average of the $U^2$ norms of derivatives of $F$, extract the maximum Fourier coefficients of $\Delta_{h_1, ..., h_{s-2}}F$, and show that for a dense proportion of $(h_1, ..., h_{s-2})$ these coefficients satisfy certain linear relations provided $||F||_{U^s}\gg p^{-c}$. If $s=3$ and $\phi(h)$ is the phase of the maximum Fourier coefficient of $\Delta_h F$, then we show that $\phi$ is constant on a dense proportion of $h$. For $s>3$, analogous relations are somewhat more complicated. These linear relations turn out to be sufficient to get a control of the $U^s$ norm of $F$ by its $U^{s-1}$ norm with polynomial bounds.

Using the definition of Gowers norms, we have
\begin{align*}
    \eta :=||F||_{U^s}^{2^s}=\EE_{h_1,...,h_{s-2}}||\Delta_{h_1,...,h_{s-2}}F||_{U^2}^4.
\end{align*}
Let $H_1=\{(h_1, ..., h_{s-2})\in\FF_p^{s-2}:||\Delta_{h_1,...,h_{s-2}}F||_{U^2}^4\geqslant\frac{1}{2}\eta\}$. To simplify the notation, let $\hh=(h_1, ...,h_{s-2})$ and $\EE_\hh:=\EE_{\hh\in\FF_p^{s-2}}$.
From the popularity principle (see e.g. Exercise 1.1.4 in \cite{tao_vu_2006}) it follows that $|H_1|\geqslant\frac{1}{2}\eta p^{s-2}$, and so
\begin{align}\label{Long expression 1}
    \frac{1}{4}\eta^2 &\leqslant\EE_{\hh}||\Delta_{\hh}F||_{U^2}^4\cdot 1_{H_1}(\hh).
\end{align} 
The $U^2$ inverse theorem, stated in Section \ref{section on Gowers norms}, implies that the square of the $U^2$ norm of a function is bounded by its maximum Fourier coefficient. Given $\Delta_{\hh} F$, let $\widehat{\Delta_{\hh} F}(\phi({\hh}))$ denote its maximum Fourier coefficient. Then the right hand side of (\ref{Long expression 1}) is bounded by
\begin{align}\label{Long expression 2}
    \nonumber
    \EE_{\hh}|\widehat{\Delta_{\hh}F}(\phi(\hh))|^2 1_{H_1}(\hh)&=\EE_\hh |\EE_x\Delta_{\hh} F(x) e_p(\phi(\hh) x)|^21_{H_1}(\hh)\\
    &=\EE_{x,x',\hh}\Delta_{\hh} F(x)\overline{\Delta_{\hh}F(x')}e_p(\phi(\hh)(x-x'))1_{H_1}(\hh).
\end{align} 
To simplify the already cumbersome notation, we denote $Q_j=P_j-P_{m+r}$ for $0\leqslant j\leqslant m+k-1$. Unpacking the definition of the dual function $F$, the expression (\ref{Long expression 2}) equals
\begin{align}\label{Long expression 3}
    &\EE_{x,x',\hh}\Delta_{\hh}\left(\EE_y\prod_{\substack{0\leqslant j \leqslant m+k-1,\\ j\neq m+r}} f_j(x+Q_j(y)) \right)\\
    \nonumber
    &\overline{\Delta_{\hh}\left(\EE_y\prod_{\substack{0\leqslant j \leqslant m+k-1,\\ j\neq m+r}} f_j(x'+Q_j(y))\right)}
    e_p(\phi(\hh)(x-x'))1_{H_1}(\hh).
\end{align}
After writing out the multiplicative derivatives, (\ref{Long expression 3}) is equal to
\begin{align}\label{Long expression 4}   
    &\EE_{x,x',\hh}\EE_{\yy,\yy'\in\FF_p^{\{0,1\}^{s-2}}}\prod_{\substack{0\leqslant j \leqslant m+k-1,\\ j\neq m+r}}\prod_{\ww\in\{0,1\}^{s-2}} \C^{|w|} f_j(x+\ww\cdot\hh+Q_j(\yy_\ww))\\ 
    \nonumber
    &\C^{|w|}\overline{f_j(x'+\ww\cdot\hh+Q_j(\yy'_\ww))}
    e_p(\phi(\hh)(x-x'))1_{H_1}(\hh).
\end{align}
The product in (\ref{Long expression 4}) contains $2^{s-2}$ copies of $f_j$ for each $j$ and each of $x$ and $x'$. In each of these copies the $y$-variable is different. We would like all the copies of $f_j$ to be expressed in terms of the same $y$-variable. To achieve this, we modify (\ref{Long expression 4}) by applying the Cauchy-Schwarz inequality $s-2$ times. First, (\ref{Long expression 4}) can be rewritten as
\begin{align}\label{Long expression 5}
    &\EE_{x,x',h_1,...,h_{s-3}}\EE_{\yy,\yy'\in\FF_p^{\{0,1\}^{s-2}}}\textbf{b}(x,x',h_1,...,h_{s-3},\yy,\yy')\EE_{h_{s-2}}\prod_{\substack{0\leqslant j \leqslant m+k-1,\\ j\neq m+r}}\prod_{\substack{\ww\in\{0,1\}^{s-2}\\ w_s=1}}\\
    \nonumber
    &\C^{|w|} f_j(x+\ww\cdot\hh+Q_j(\yy_\ww))\C^{|w|}\overline{f_j(x'+\ww\cdot\hh+Q_j(\yy'_\ww))}
    e_p(\phi(\hh)(x-x'))1_{H_1}(\hh).
\end{align}
By the Cauchy-Schwarz inequality and change of variables, (\ref{Long expression 5}) is bounded by
\begin{align}\label{Long expression 6}
    (&\EE_{x,x',h_1,...,h_{s-3}, h_{s-2},h'_{s-2}}\EE_{\yy,\yy'\in\FF_p^{\{0,1\}^{s-2}}}\prod_{\substack{0\leqslant j \leqslant m+k-1,\\ j\neq m+r}}\prod_{\substack{\ww\in\{0,1\}^{s-2}\\ w_s=1}}\C^{|w|}\\
    \nonumber
    & (f_j(x+\sum_{i=1}^{s-3}w_i h_i+w_{s-2} h_{s-2}+Q_j(\yy_\ww))\overline{f_j(x+\sum_{i=1}^{s-3}w_i h_i+w_{s-2} h'_{s-2}+Q_j(\yy_\ww))}
    \\\nonumber
    &\overline{f_j(x'+\sum_{i=1}^{s-3}w_i h_i+w_{s-2} h_{s-2} +Q_j(\yy'_\ww))}f_j(x'+\sum_{i=1}^{s-3}w_i h_i+w_{s-2} h'_{s-2} +Q_j(\yy'_\ww)))\\\nonumber
    &e_p((\phi(h_1,...,h_{s-3},h_{s-2})-\phi(h_1,...,h_{s-3},h'_{s-2}))(x-x'))\\\nonumber
    &1_{H_1}(h_1,...,h_{s-3},h_{s-2})
    1_{H_1}(h_1,...,h_{s-3},h'_{s-2}))^\frac{1}{2}.
\end{align}
The presence of so many terms in (\ref{Long expression 6}) comes from the fact that in the process of applying the Cauchy-Schwarz inequality and changing variables, each expression $E(h_{s-2})$ (depending possibly on other variables as well) is replaced by ${E(h_{s-2})\overline{E(h'_{s-2})}}$. Therefore the number of expressions in the product doubles, making     (\ref{Long expression 6}) rather lengthy. Applying Cauchy-Schwarz another $s-3$ times to $h_{s-3}, ..., h_1$ respectively, we bound (\ref{Long expression 6}) by
\begin{align}\label{intermediate long expression 1}
    &(\EE_{x,x',y,y',\hh,\hh'}\prod_{\substack{0\leqslant j \leqslant m+k-1,\\ j\neq m+r}}\prod_{\ww\in\{0,1\}^{s-2}}(1_{H^1}(\hh^{(\ww)}) \C^{|w|}f_j(x+\ww\cdot\hh^{(\ww)}+Q_j(y))\\
    \nonumber
    &\overline{\C^{|w|}f_j(x'+\ww\cdot\hh^{(\ww)}+Q_j(y'))})e_p((\sum_{\ww\in\{0,1\}^{s-2}}(-1)^{|w|}\phi(\hh^{(\ww)})(x-x')))^{\frac{1}{2^{s-2}}}.
\end{align}
where \[ \hh^{(\ww)}_i =
\begin{cases}
h_i, w_i = 0\\
h'_i, w_i = 1
\end{cases}
\]
The expression (\ref{intermediate long expression 1}) can be simplified to
\begin{align*}
    \left(\EE_{\hh,\hh'}\left|\EE_{x,y}\left(\prod_{\substack{0\leqslant j \leqslant m+k-1,\\ j\neq m+r}} g_j(x+P_j(y))\right)e_p \left(\psi(\hh,\hh')(x+P_{m+r}(y))\right)
    \right|^2\1_{\square(H_1)}(\hh,\hh')\right)^\frac{1}{2^{s-2}}
\end{align*}
where
\begin{align*}
    g_j(t):=\prod_{\ww\in\{0,1\}^{s-2}}\C^{|w|}f_j(t+\ww\cdot\hh^{(\ww)}),
\end{align*}
\begin{align*}
    \square(A):=\{(\hh,\hh')\in\FF_p^{2(s-2)}:\forall \ww\in\{0,1\}^{s-2}\hh^{(\ww)}\in A\}
\end{align*}
and
\begin{align*}
    \psi(\hh,\hh'):=\sum_{\ww\in\{0,1\}^{s-2}}(-1)^{|w|}\phi(\hh^{(\ww)}).
\end{align*}
Recall that for $m+r+1\leqslant j\leqslant m+k-1$, we have defined $f_j$ to be $f_j(x) = e_p(a_j x)$. Combining this with the assumption that $s>2$, we have that
\begin{align*}
    g_j(x+P_j(y))=e_p\left(a_j\sum_{\ww\in\{0,1\}^{s-2}}(-1)^{|w|}\ww\cdot\hh^{(\ww)}\right)
\end{align*}
for these values of $j$. This expression depends only on $\hh$ but not on $x$ or $P_j$, and so we incorporate $g_{m+r+1}$, ..., $g_{m+k-1}$ into the absolute value.
We thus obtain the estimate
\begin{align}\label{important intermediate expression}
    \EE_{\hh,\hh'}&\left|\EE_{x,y}\left(\prod_{\substack{0\leqslant j \leqslant m+r-1}} g_j(x+P_j(y))\right)e_p \left(\psi(\hh,\hh')(x+P_{m+r}(y)))\right)
    \right|^2\\\nonumber
    &\1_{\square(H_1)}(\hh,\hh')\geqslant\left(\frac{\eta}{2}\right)^{2^{s-1}}.
\end{align}

We are now able to apply the induction hypothesis. By $\SSS(m,k,r)$, the expression inside the absolute values equals $O(p^{-c})$ unless $\psi(\hh,\hh') = 0$. Therefore, the set
\begin{align*}
    H_2:=\left\{(\hh,\hh')\in \square(H_1): \psi(\hh,\hh') = 0\right\}
\end{align*}
has size at least 
\begin{align*}
    \left(\left(\frac{\eta}{2}\right)^{2^{s-1}}-O(p^{-c})\right)p^{2(s-2)}.
\end{align*}
In particular, there exists $\hh\in H_1$ such that the fiber
\begin{align*}
    H_3:=\{\hh': (\hh,\hh')\in H_2\}
\end{align*}
has size at least 
\begin{align*}
    \left(\left(\frac{\eta}{2}\right)^{2^{s-1}}-O(p^{-c})\right)p^{s-2}.
\end{align*}
Fix this $\hh$. We now show that the phases $\phi$ possess some amount of low-rank structure which we subsequently use to complete the proof of the lemma. By the definitions of $H_2$ and $H_3$, for each $\hh'\in H_3$ we have $\psi(\hh,\hh')=0$. Define
\begin{align*}
    \psi_i(\hh,\hh'):=(-1)^s\sum_{\substack{\ww\in\{0,1\}^{s-2},\\ w_1 = ... = w_{i-1}=1,\\ w_i = 0}}(-1)^{|w|}\phi(\hh^{(\ww)}).
\end{align*}
Note that, $\psi(\hh,\hh')=\phi(h'_1,...,h'_{s-2})-\psi_1(\hh,\hh')-...-\psi_{s-2}(\hh,\hh')$. Crucially, $\psi_i$ does not depend on $h'_1, ..., h'_i$. Thus, $\psi(\hh,\hh') = 0$ implies that
\begin{align*}
    \phi(h'_1,...,h'_{s-2})=\sum_{i=1}^{s-2}\psi_i(\hh,\hh').
\end{align*}
That is to say, $\phi(h'_1, ...,h'_{s-2})$ can be decomposed into a sum of $s-2$ functions, each of which does not depend on $h_i'$ for a different $i$.

To alleviate the pain that the reader may experience while struggling with the notation, we illustrate the aforementioned for $s=3$ and 4. For $s=3$, 
$$\psi(h,h')=\phi(h)-\phi(h')=\psi_1(h)-\phi(h').$$ Hence $\psi(h,h')=0$ implies that $\phi(h')=\phi(h)$. For $s=4$, 
\begin{align*}
    \psi(\hh,\hh') &=\phi(h_1,h_2)-\phi(h_1',h_2)-\phi(h_1,h_2')+\phi(h_1',h_2')\\
    &=\psi_1(\hh,\hh')-\psi_2(\hh,\hh')+\phi(h_1',h_2')
\end{align*}
 and so $\psi(\hh,\hh')=0$ implies that
\begin{align*}
    \phi(h_1',h_2') = \phi(h_1,h_2')+\phi(h_1',h_2)-\phi(h_1,h_2) = \psi_2(\hh,\hh')-\psi_1(\hh,\hh').
\end{align*}

We now estimate the expression
\begin{align}\label{Long expression 7}
    \EE_{\hh'}||\Delta_{\hh'}F||_{U^2}^4\1_{H_3}(\hh')
\end{align}
from above and below. From below, it is bounded by $$\frac{\eta}{2}\cdot \left(\left(\frac{\eta}{2}\right)^{2^{s-1}}-O(p^{-c})\right)\geqslant \left(\frac{\eta}{2}\right)^{2^{s}}-O(p^{-c}).$$
The upper bound is more complicated, and it relies on the fact that we can decompose $\phi(\hh')$ into a sum of $\psi_i$'s such that $\psi_i$ does not depend on $h'_i$. Using $U^2$-inverse theorem, (\ref{Long expression 7}) is bounded from above by:
\begin{align}\label{Long expression 8}
    \EE_{\hh'}\left|\widehat{\Delta_{\hh'}F}(\phi(\hh'))\right|^2\1_{H_3}(\hh') &= \EE_{\hh'}\left|\widehat{\Delta_{\hh'}F}\left(\sum_{i=1}^{s-2}\psi_i(\hh')\right)\right|^2\1_{H_3}(\hh').
\end{align}
By positivity, we can extend (\ref{Long expression 8}) to the entire $\FF_p^{s-2}$; that is, we have
\begin{align}
    \EE_{\hh'}\left|\widehat{\Delta_{\hh'}F}\left(\sum_{i=1}^{s-2}\psi_i(\hh')\right)\right|^2\1_{H_3}(\hh')
    \leqslant \EE_{\hh'}\left|\widehat{\Delta_{\hh'}F}\left(\sum_{i=1}^{s-2}\psi_i(\hh')\right)\right|^2.
\end{align}
Rewritting, we obtain that
\begin{align}\label{Long expression 9}
    \nonumber
    \EE_{\hh'}\left|\widehat{\Delta_{\hh'}F}\left(\sum_{i=1}^{s-2}\psi_i(\hh')\right)\right|^2 &= \EE_{\hh'}\left|\EE_x\Delta_{\hh'}F(x)e_p\left(\sum_{i=1}^{s-2}\psi_i(\hh')x\right)\right|^2\\
    &=\EE_{x,\hh',h_{s-1}}\Delta_{\hh',h_{s-1}}F(x)e_p\left(\sum_{i=1}^{s-2}\psi_i(\hh')h_{s-1}\right).
\end{align}
We apply Cauchy-Schwarz $s-2$ times to (\ref{Long expression 9}) to get rid of the phases $\psi_i(\hh')$. In the first application, we start by rewriting (\ref{Long expression 9}) as
\begin{align}\label{Long expression 10}
\EE_{\substack{x,h_2',...,\\ h'_{s-2},h_{s-1}}}\textbf{b}(x, h_2', ..., h'_{s-2}, h_{s-1})\EE_{h'_1}\Delta_{\substack{h'_2,...,h'_{s-2},h_{s-1}}}F(x+h'_1)e_p\left(\sum_{i=2}^{s-2}\psi_i(\hh')h_{s-1}\right)
\end{align}
and then we bound it by
\begin{align}\label{Long expression 11}
&(\EE_{x,h'_1, h''_1, h_2',..., h'_{s-2},h_{s-1}}\Delta_{h'_2,...,h'_{s-2},h_{s-1}}\left(F(x+h'_1)\overline{F(x+h''_1)}\right)\\
\nonumber
&e_p\left(\sum_{i=2}^{s-2}(\psi_i(h'_1,h'_2,...,h'_{s-2})-\psi_i(h''_1,h'_2...,h'_{s-2}))h_{s-1}\right))^{\frac{1}{2}}.
\end{align}
After repeatedly applying Cauchy-Schwarz in this manner, we get rid of all the phases and bound (\ref{Long expression 11}) by $||F||^2_{U^{s-1}}$. This proves the lemma.
\end{proof}

The second proof is simpler.
\begin{proof}[Proof of Lemma \ref{$U^1$ control of the dual}]

By $U^2$-inverse theorem,
$||F||^2_{U^2}\leqslant\max\limits_{\alpha\in\FF_p}|\hat{F}(\alpha)|$.    
By Lemma \ref{Vanishing of expression II}, this is $O(p^{-c})$ unless $\alpha = 0$, in which case $$\hat{F}(\alpha)=\Lambda_{m, P_m, ..., P_{m+r-1}, P_{m+r+1}, ..., P_{m+k-1}}(f_0,f_1,..., f_{m+r-1}, f_{m+r+1}, ...,f_{m+k-1}).$$ Thus,
\begin{align}\label{inequality in Lemma 9}
||F||^2_{U^2} &\leqslant 1_{\alpha=0}|\Lambda_{m, P_m, ..., P_{m+r-1}, P_{m+r+1}, ..., P_{m+k-1}}(f_0,f_1,..., f_{m+r-1}, f_{m+r+1}, ...,f_{m+k-1})| + O(p^{-c})\\  
\nonumber
&\leqslant|\Lambda_m(f_0, ..., f_{m-1})|\prod_{\substack{m\leqslant j\leqslant m+k-1,\\ j\neq m+r}}|\EE f_j| + O(p^{-c})
\end{align}
where the second inequality follows from applying $\SSS(m,k-1)$. Recalling that $f_j(t)=e_p(a_j t)$ for $m+r+1\leqslant j\leqslant m+k-1$ and combining it with (\ref{inequality in Lemma 9}), we get that
\begin{align*}
    ||F||^2_{U^2} \leqslant\min_{m\leqslant j\leqslant m+r-1}||f_j||_{U^1}\cdot\prod_{j=m+r+1}^{m+k-1}1_{a_j=0}+O(p^{-c}).
\end{align*}
Taking square roots on both sides and applying H\"{o}lder's inequality proves the lemma.
\end{proof}

Next we prove Lemma \ref{$U^1$ control of the operator} using the previous lemmas. 
\begin{proof}[Proof of Lemma \ref{$U^1$ control of the operator}]

Take $s=s_0$ and $\beta$ for which Lemma \ref{PET induction} holds. Using Lemma \ref{Hahn-Banach decomposition}, we decompose $f_{m+r}$ into
\begin{align*}
    f_{m+r} = f_a + f_b + f_c
\end{align*}
with $||f_a||_{U^{s_0}}^*\leqslant p^{\delta_1}$, $||f_b||_{L^1}\leqslant p^{-\delta_2}$, $||f_c||_{L^\infty}\leqslant p^{\delta_3}$, $||f_c||_{U^{s_0}}\leqslant p^{-\delta_4}$, and split the $\Lambda$ operator accordingly. The values of the parameters $\delta_1,\delta_2,\delta_3,\delta_4$ have to satisfy (\ref{delta inequality}) and will be determined later. The term involving $f_b$ is easy to bound using H\"{o}lder inequality
\begin{align*}
    |\langle F,f_b\rangle|\leqslant||F||_{L^\infty}||f_b||_{L^1}\leqslant p^{-\delta_2}.
\end{align*}
The term involving $f_c$ can also be bounded from above provided $\delta_4$ is sufficiently large compared to $\delta_3$
\begin{align*}
    |\langle F, \overline{f_c}\rangle| &=||f_c||_{L^\infty}\left|\left\langle F,\frac{\overline{f_c}}{||f_c||_{L^\infty}}\right\rangle\right|\\
    &\leqslant p^{\delta_3}\left(\left(\frac{p^{-\delta_4}}{p^{\delta_3}}\right)^\beta+ O(p^{-\beta})\right)\\
    &\ll p^{\delta_3(1-\beta)-\beta\delta_4}+p^{\delta_3-\beta}
\end{align*}
where in the second inequality we are using Lemma \ref{PET induction}.
Finally, the term involving $f_a$ can be bounded using dual inequality
\begin{align*}
    |\langle F, \overline{f_a}\rangle| &\leqslant||f_a||^*_{U^{s_0}}||F||_{U^{s_0}}\leqslant p^{\delta_1}||F||_{U^{s_0}}.
\end{align*}
Using the decomposition, we obtain the following bound on $\Lambda$ in terms of the $U^{s_0}$ norm of the dual function $F$
\begin{align*}
    |\Lambda(f_0,...,f_{m+k-1})| &\leqslant|\langle F, \overline{f_a}\rangle| + |\langle F, \overline{f_b}\rangle| + |\langle F, \overline{f_c}\rangle|\\
    &\leqslant p^{\delta_1}||F||_{U^{s_0}}+p^{-\delta_2} +p^{\delta_3(1-\beta)-\beta\delta_4}+p^{\delta_3-\beta}.
\end{align*}
From Lemma \ref{degree lowering} it follows that 
\begin{align*}
    ||F||_{U^2}= \Omega(||F||_{U^{s_0}}^{2^{(s_0-2)(s_0+2)}})-O(p^{-c}).
\end{align*}
Let $s_1=(s_0-2)(s_0+2)$. We thus have that
\begin{align*}
    |\Lambda(f_0,... ,f_{m+k-1})|\ll p^{\delta_1}||F||_{U^2}^{2^{-s_1}}+p^{\delta_1-2^{-s_1}c}+p^{-\delta_2}+ p^{\delta_3(1-\beta)-\beta\delta_4}+p^{\delta_3-\beta}.
\end{align*}
Using Lemma \ref{$U^1$ control of the dual}, we establish a $U^1$ control by $f_m, ..., f_{m+r-1}, f_{m+r+1}, ..., f_{m+k-1}$
\begin{align}\label{expression in proof of lemma 9}
    |\Lambda(f_0, ..., f_{m+k-1})| &\ll p^{\delta_1}\min_{\substack{m\leqslant j\leqslant m+r-1}}||f_j||_{U^1}^{2^{-s_1-1}}\cdot\prod_{j=m+r+1}^{m+k-1}1_{a_j=0}+p^{\delta_1-2^{-s_1}c}\\
    \nonumber&+p^{-\delta_2} + p^{\delta_3(1-\beta)-\beta\delta_4}+p^{\delta_3-\beta}.
\end{align}
Let $c_0$ be the value of $c$ appearing in (\ref{expression in proof of lemma 9}).
Setting the values of the parameters to be
\begin{align*}
    \delta_1 = 2^{-s_1}\frac{c_0}{2},\quad    
    \delta_2 = \beta2^{-s_1}\frac{c_0}{8},\quad
    \delta_3 &= \beta2^{-s_1}\frac{c_0}{4},\quad \rm{and} \quad
    \delta_4 = (1-\beta)2^{-s_1}\frac{c_0}{2}
\end{align*}
proves the lemma.
\end{proof}


\section{Upper bounds for subsets of $\FF_p$ lacking arithmetic progressions with $k$-th power common differences}\label{section on upper bounds for subsets of finite fields lacking arithmetic progressions with restricted differences}

We now switch gears, moving away from the progression (\ref{generalized union of AP and GP}) towards arithmetic progressions with common difference coming from the set of $k$-th powers. In this section, we prove Theorem \ref{Sets lacking arithmetic progressions with $k$-th power differences} assuming Theorem \ref{counting theorem for linear forms with restricted variables}. The argument goes much the same way as deriving Theorem \ref{main theorem, general case} from Theorem \ref{counting theorem for the more difficult configuration}.

First, we prove the following simple lemma which allows us to reduce to the case $k| p-1$.
\begin{lemma}\label{Q_k = Q_k'}
Let $k\in\NN_+$ and $Q_k$ be the set of $k$-th power residues in $\FF_p$. Then $Q_k=Q_{\gcd(k,p-1)}$.
\end{lemma}

\begin{proof}
Since $\FF_p^\times$ is a cyclic group under multiplication, we can write it as $\FF_p^\times=\langle a| a^{p-1}=1\rangle$. Note that for each $k\in\NN$, $Q_k$ and $Q_{\gcd(k,p-1)}$ are subgroups of $\FF_p^\times$ of cardinality $\frac{p-1}{\gcd(k,p-1)}$, generated respectively by $a^k$ and $a^{\gcd(k,p-1)}$. The property $\gcd(k,p-1)|k$ moreover implies that $Q_k$ is a subgroup of $Q_{\gcd(k,p-1)}$, and so they must be equal.
\end{proof}

\begin{proof}[Proof of Theorem \ref{Sets lacking arithmetic progressions with $k$-th power differences}]
The set of $k$-th powers in $\FF_p$ is precisely $Q_k$, and by Lemma \ref{Q_k = Q_k'} it is the same as the set $Q_{\gcd(k,p-1)}$. Therefore we can assume that $k$ divides $p-1$, otherwise we replace $k$ with $\gcd(k,p-1)$. Suppose $A\subset\FF_p$ for $p\geqslant p_0$ of size $|A|=\alpha p$ lacks $m$-term arithmetic progressions with difference coming from the set of $k$-th powers. From Theorem \ref{counting theorem for linear forms with restricted variables} it follows that
\begin{align}\label{Counting equation for APs}
&\EE_{x,y}\1_A(x)\1_A(x+y)...\1_A(x+(m-1)y)\1_{Q_k}(y)\\
&=\frac{1}{k}\EE_{x,y}\1_A(x)\1_A(x+y)...\1_A(x+(m-1)y)+O\left(p^{-c}\right).\nonumber
\end{align}

Since $A$ lacks progressions with $k$-th power differences, the left-hand side of (\ref{Counting equation for APs}) is 0, and so we have
\begin{align}\label{Counting equation for APs, II}
    \EE_{x,y}\1_A(x)\1_A(x+y)...\1_A(x+(m-1)y)=O\left(p^{-c}\right).
\end{align}
Applying Lemma \ref{Averaging over progressions} to (\ref{Counting equation for APs}) gives $M^{-2}\ll p^{-c}$ where $M=M_m(\frac{1}{2}\alpha)$ and $M_m$ is the inverse function to $s_m$ on $(0,\alpha_0]$, $\alpha_0 = s_m(p_0)$. 
Since $M$ grows faster than polynomially in $\alpha^{-1}$ by Behrend's construction \cite{behrend_1946}, this gives $M_m\gg p^c$. Applying $s_m$ to both sides and noting that $s_m$ is decreasing, we obtain that $\alpha\leqslant 2s_m( Cp^c)$.

\end{proof}

\section{Counting theorem for the number of linear configurations in subsets of $\FF_p$ with variables restricted to the set of $k$-th powers}

This section is devoted to the proof of Theorem \ref{counting theorem for linear forms with restricted variables}. We will first show that without loss of generality, we can assume that $k_i$ divides $p-1$ for each $1\leqslant i\leqslant d$. This will simplify the notation in the rest of the argument. 

\begin{lemma}\label{Rewriting the counting operator}
We have 
\begin{align*}
    \EE_{x_1,...,x_d}\prod\limits_{i=1}^m f_j(L_i(x_1^{k_1},...,x_d^{k_d}))&=\EE_{x_1,...,x_d}\prod\limits_{i=1}^m  f_j(L_i(x_1^{k'_1},...,x_d^{k'_d}))\\
    &={k'_1...k'_d}\EE_{x_1,...,x_d}\prod\limits_{i=1}^m  f_j(L_i(x_1,...,x_d))\prod_{i=1}^d 1_{Q_{k'_i}}(x_i)+O\left(p^{-1}\right)
\end{align*} where $k_i':=\gcd(k_i,p-1)$ for each $1\leqslant i\leqslant d$. 
\end{lemma}
\begin{proof}
By Lemma \ref{Q_k = Q_k'}, $Q_{k}=Q_{\gcd(k,p-1)}$ for each $k\in\NN_+$. Therefore the set of $k_i$-th power residues agrees with the set of $k'_i$-th power residues for each $1\leqslant i\leqslant d$. Consequently, the set of tuples $$\{(x_1^{k_1},...,x_d^{k_d}): (x_1,...,x_d)\in\FF_p^d\}$$ equals the set of tuples $$\{(x_1^{k'_1},...,x_d^{k'_d}): (x_1,...,x_d)\in\FF_p^d\},$$
and moreover each tuple $(x_1^{k_1},...,x_d^{k_d})$ appears in $\FF_p^d$ the same number of times as the tuple $(x_1^{k'_1},...,x_d^{k'_d})$. This implies the first equality, as the summations in both expressions are carried over the same sets of tuples the same number of times.

The second equality follows from the fact that each value of $y\in \FF_p^\times$ equals $x_i^{k'_i}$ for precisely $k'_i$ different values of $x_i\in\FF_p$. The error term $O\left(p^{-1}\right)$ corresponds to the cases when at least one of the variables $x_1,...,x_d$ is 0. Using union bound, there are at most $dp^{d-1}$ such cases, which together contribute at most $\frac{d}{p}$ to the expectation.
\end{proof}

We thus assume for the rest of this section that $k_1, ..., k_d$ are coprime to $p-1$. With this assumption, we now describe a useful expression for $\1_{Q_k}$ which is crucial in proving the error term in Theorem \ref{counting theorem for linear forms with restricted variables}. Let $a$ be a generator for the multiplicative group $\FF_p^\times$. Define the map
\begin{align*}
\chi_k: \FF_p^\times &\to \CC\\
a^l &\mapsto e_k(l).
\end{align*}
The function $\chi_k$ is thus a \emph{multiplicative character of order $k$}, i.e. a group homomorphism from $\FF_p^\times$ to $\CC^\times$ satisfying $\chi_k^k=1$. We extend $\chi_k$ to $\FF_p$ by setting $\chi_k(0)=0$. Then $\chi_k$ picks out $Q_k$, in the sense that $\chi_k(x)=1\iff$ $x\in Q_k$. Using the orthogonality of roots of unity, we can write
\begin{align}
\label{sum of characters}
    \1_{Q_k}(x)=\frac{1+\chi_k(x)+\chi_k(x)^2+...+\chi_k(x)^{k-1}}{k}-\frac{1}{k}\1_{\{0\}}(x).
\end{align}

We now use (\ref{sum of characters}) to replace each $\1_{Q_{k_i}}$ by a sum of characters in (\ref{main equation}). Using the multilinearity of the operator, we obtain a main term of the same form as in (\ref{main equation}), which corresponds to the terms in (\ref{sum of characters}) having $\1_{Q_{k_i}}$ replaced by $\frac{1}{k_i}$. Terms where $\1_{Q_{k_i}}$ is replaced by $\frac{1}{k_i}\1_{\{0\}}(x)$ are of size $O\left(p^{-1}\right)$, and there is a bounded number of them. It remains to deal with the terms that contain some $\frac{\chi^j_k(x)}{k}$ with $j>0$ but have no $\frac{1}{k_i}\1_{\{0\}}(x)$. Each such term is of the form 
\begin{align}\label{error term}
    \EE_{x_1,...,x_d}\prod\limits_{i=1}^m  f_j(L_i(x_1,...,x_d))\prod_{i\in S}\frac{\chi_{k_i}^{j_i}(x_i)}{k_i}
\end{align}
for a nonempty $S\subset\{1\leqslant i\leqslant d: k_i>1\}$ and $1\leqslant j_i\leqslant k_i-1$. From the fact that $k_i$ divides $d$ it follows that $\chi_{k_i}^{j_i}$ is also a character of order $k_i$, so without loss of generality we can take $j_i=1$ for each $1\leqslant i\leqslant d$.

Green and Tao proved that linear forms $L'_1(x_1, ..., x_d)$, ..., $L'_m(x_1, ..., x_d)$ are controlled by a Gowers norm \cite{green_tao_2010, tao_2012}: specifically, they showed that
\begin{align}\label{linear configurations 2}
    \left|\EE_{x_1,...,x_d}\prod_{j=1}^m g_j(L'_i(x_1,...,x_d))\right|\leqslant\min\limits_{1\leqslant j\leqslant m}||f_j||_{U^s}
\end{align}
whenever for each $1\leqslant i\leqslant m$ one can partition $\{L'_j: j\neq i\}$ into $s+1$ classes such that $L'_i$ does not lie in the span of each of them. The lowest $s-1$ for which this is true is called \emph{Cauchy-Schwarz complexity}, or \emph{CS-complexity} of the system of linear forms $L'_1, ..., L'_m$. The only case when such $s$ may not exist is if two linear forms $L'_i$ and $L_j'$ are the same up to scaling. Otherwise we can partition linear forms into such classes: in the worst case, each of $\{L'_j: j\neq i\}$ forms a separate class, in which case the CS-complexity is $m-2$. This extreme case occurs in arithmetic progressions, for instance: the operator
\begin{align*}
    \EE_{x,y}f_0(x)f_1(x+y)...f_{m-1}(x+(m-1)y)
\end{align*}
is bounded by $||f_{i}||_{U^{m-1}}$ for each $0\leqslant i\leqslant m-1$, and the system of linear forms $\{x, x+y, ..., x+(m-1)y\}$ has CS-complexity $m-2$. 

We assumed specifically that no two linear forms $L_i$, $L_j$ are scalar multiples, and that $L_i$ is never a scalar multiple of $e_j$. From these assumptions we obtain the following lemma, which is essentially a restatement of Green and Tao's result tailored to our context.
\begin{lemma}\label{bounding the error term}
For an arbitrary character $\chi_{k_i}$ of order $k_i$, we have the bound
\begin{align}\label{U^s control}
    \left|\EE_{x_1,...,x_d}\prod\limits_{j=1}^m  f_j(L_i(x_1,...,x_d))\prod_{i\in S}\frac{\chi_{k_i}(x_i)}{k_i}\right|\leqslant\left(\prod_{i\in S}\frac{1}{k_i}\right)\min_{i\in S}||\chi_{k_i}||_{U^s}
\end{align}
where $s-1$ is the CS-complexity of the system
\begin{align}\label{specific system of linear forms}
\{L_1, ..., L_m\}\cup\{x_j: j\in S\}
\end{align}
In particular, one can take $s=m+|S|-1\leqslant m+d-1$.

\end{lemma}
\begin{proof}
By assumption, all forms in the system
\begin{align}\label{the largest system of linear forms}
\{L_1, ..., L_m\}\cup\{x_j: 1\leqslant j\leqslant d, k_j >1\}    
\end{align}
are pairwise linearly independent. Since (\ref{specific system of linear forms}) is a subset of (\ref{the largest system of linear forms}), all forms in (\ref{specific system of linear forms}) are also pairwise linearly independent. Therefore the CS-complexity of this system is finite, and is at most $m+|S|-2$ because the system (\ref{specific system of linear forms}) consists of $m+|S|-1$ linear forms.
\end{proof}

It thus follows that the error term in (\ref{main equation}) is controlled by Gowers norms of characters. The multiplicative property of characters makes it easy to bound their Gowers norms using tools such as Weil's bound.

\begin{lemma}[Weil's bound]
Let $\chi$ be a nonprincipal multiplicative character of $\FF_p$ of order $k$, and let $P\in\FF_p[x]$ be a polynomial with $r$ distinct roots in the splitting field. If $P$ is not a $k$-th power, then
\begin{align*}
    \left|\EE_x\chi(P(x))\right|\leqslant(r-1)q^{-\frac{1}{2}}.
\end{align*}
\end{lemma}

In particular, we use the following corollary, which is Corollary 11.24 in Iwaniec \& Kowalski \cite{iwaniec_kowalski_2004}.
\begin{lemma}[Corollary to Weil's bound]\label{Corollary to Weil's bound}
Let $\chi$ be a nonprincipal multiplicative character of $\FF_p$, and let $b_1,...,b_{2r}\in\FF_p$. If one of them is different from the others, then
\begin{align*}
    \left|\EE_x\chi((x-b_1)...(x-b_r))\overline{\chi}((x-b_{r+1})...(x-b_{2r}))\right|\leqslant2r p^{-\frac{1}{2}}.
\end{align*}
\end{lemma}

With this corollary, we can easily estimate the Gowers norms of nonprincipal multiplicative characters.
\begin{lemma}[Gowers norms of characters]\label{Gowers norms of characters}
If $\chi$ is a nonprincipal multiplicative character of $\FF_p$ of order $k$ and $s$ is a natural number, then 
\begin{align*}
    ||\chi||_{U^s}\leqslant 2p^{-2^{-(s+1)}}.
\end{align*}
\end{lemma}
The reader may also consult \cite{fouvry_kowalski_michel_2013} for a more general discussion of Gowers norms of functions on $\FF_p$ of a strongly algebraic nature.
\begin{proof}
By definition, the $U^s$ norm of $\chi$ is given by the following expression
\begin{align*}
    ||\chi||_{U^s}^{2^s} &= \EE_{h_1,...,h_s}\EE_x \prod_{\ww\in\{0,1\}^s}C^{|w|}\chi(x+\ww\cdot \hh)\\
    &= \EE_{h_1,...,h_s}\EE_x \chi\left(\prod_{\ww\in\{0,1\}^s, |w|\;\rm{even}}(x+\ww\cdot \hh)\right)\overline{\chi}\left(\prod_{\ww\in\{0,1\}^s, |w|\;\rm{odd}}(x+\ww\cdot \hh)\right)\\
    &\leqslant\EE_{h_1,...,h_s}\left|\EE_x \chi\left(\prod_{\ww\in\{0,1\}^s, |w|\;\rm{even}}(x+\ww\cdot \hh)\right)\overline{\chi}\left(\prod_{\ww\in\{0,1\}^s, |w|\;\rm{odd}}(x+\ww\cdot \hh)\right)\right|.
\end{align*}
If $\ww\cdot \hh$ are not all equal, then by Lemma \ref{Corollary to Weil's bound} we have
\begin{align*}
    \left|\EE_x \chi\left(\prod_{\ww\in\{0,1\}^s, |w|\;\rm{even}}(x+\ww\cdot \hh)\right)\overline{\chi}\left(\prod_{\ww\in\{0,1\}^s, |w|\;\rm{odd}}(x+\ww\cdot \hh)\right)\right|\leqslant 2^s p^{-\frac{1}{2}}.
\end{align*}

The only possibility for $\ww\cdot \hh$ being equal for all $\ww\in\{0,1\}^s$ is when $h_1=...=h_s=0$, which happens with probability $p^{-s}$. Thus
\begin{align*}
    ||\chi||_{U^s}^{2^s}\leqslant 2^s p^{-\frac{1}{2}}+p^{-s}
\end{align*}
and so
\begin{align*}
    ||\chi||_{U^s}\ll p^{-2^{-(s+1)}}.
\end{align*}
\end{proof}

Applying the results of Lemma \ref{Gowers norms of characters} to Lemma \ref{bounding the error term}, we see that the error term in (\ref{main equation}) is of the size $O\left(p^{-c}\right)$, which proves Theorem \ref{counting theorem for linear forms with restricted variables}.

\section{Further discussion}
There are many directions in which one could try to extend the results of this paper, in particular Theorem \ref{counting theorem for the more difficult configuration}. One of the questions one might ask is whether there is a discorrelation result for progressions of the form
\begin{align*}
    x, x+Q(y), ..., x+(m-1)Q(y), x+P_m(y), ..., x+P_{m+k-1}(y)
\end{align*}
where $Q$ has degree greater than 1 while $P_m, ..., P_{m+k-1}$ are linearly independent and presumably satisfy a further technical assumption of algebraic independence similar to one in Theorem \ref{counting theorem for the more difficult configuration}. Combining methods used in the proofs of Theorems \ref{counting theorem for the more difficult configuration} and \ref{counting theorem for linear forms with restricted variables}, one can easily derive a statement of the form:
\begin{theorem}\label{counting theorem for the more difficult configuration with y^k}
Let $m,k,l\in\NN_+$ and $P_m$, ..., $P_{m+k-1}$ be polynomials in $\ZZ[y]$ such that $$a_m P_m + ... + a_{m+k-1}P_{m+k-1}$$ has degree at least $m$ unless $a_m = ... = a_{m+k-1}=0$ (in particular, $P_m$, ..., $P_{m+k-1}$ are linearly independent and each of them has degree at least $m$). Suppose $f_0, ..., f_{m+k-1}$ are 1-bounded functions from $\FF_p$ to $\CC$. Then
\begin{align}
    &\EE_{x,y}\prod_{j=0}^{m-1}f_j(x+jy^l)\prod_{j=m}^{m+k-1}f_j(x+P_j(y^l))\\
    \nonumber
    &=\left(\EE_{x,y}\prod_{j=0}^{m-1}f_j(x+jy)\right)\prod_{j=m}^{m+k-1}\EE f_j + O(p^{-c})
\end{align}
where all the constants are positive and depend on $m, k, l$ and polynomials $P_m, ..., P_{m+k-1}$ but not on $f_0, ..., f_{m+k-1}$.
\end{theorem}
This is a version of Theorem \ref{counting theorem for the more difficult configuration} where variable $y$ is restricted to lie in the set of $l$-th powers. It essentially says that restricting the variables to the set of $l$-th powers does not matter. For instance, this theorem allows us to prove that a set $A\subset\FF_p$ lacking progressions of the form
\begin{align*}
    x, x+y^l, ..., x+(m-1)y^l, x+y^{ml}, ..., x+y^{(m+k-1)l}
\end{align*}
has size at most
\[ |A|\ll\begin{cases}
p^{-c},\; &m = 1,2,\\
p\frac{(\log \log p)^4}{\log p},\; &m = 3,\\
\frac{p}{(\log p)^{c}},\; &m = 4,\\
\frac{p}{(\log\log p)^{c}}, \; &m>4
\end{cases}
\]
where the implied constant depends on $k, m, l$ and $c$ depends on $m$ only. Note that the bounds here are of the same shape as the bounds in Theorem \ref{main theorem, special case}: this is because the proof of this corollary is identical to the proof of Theorem \ref{main theorem, general case}. 

The drawback of this theorem is that it essentially only works for polynomials $P_m', ..., P'_{m+k-1}$ that can be expressed as polynomials in $y^l$, i.e. $P'_i(y)=P_j(y^l)$ for some $P_j$. For instance, it allows us to handle $$x, x+y^2, x+2y^2, x+y^6$$ but not 
\begin{align*}
    x,\; x+y^2,\; x+2y^2,\; x+y^5 \quad {\rm{or}} \quad x,\; x+y^2,\; x+2y^2,\; x+y^5.
\end{align*}
Replacing $P_m(y^l), ..., P_{m+k-1}(y^l)$ in the statement of the theorem by $P_m(y)$, ..., $P_{m+k-1}(y)$ would require a completely different approach. We have an argument that would allow us to replace $P_j(y^l)$ by $P_j(y)$ for $m=3$ and possibly $m=4$, however it has two serious downsides. First, the argument only works if the minimal degree of $P_j$'s is unreasonably large depending on $m$ and $l$ - it in fact would have to be greater than the minimal value $s$ obtained by applying Lemma \ref{PET induction} to $x,\; x+y^l,\; ...,\; x+(m-1)y^l,\; x+P_m(y),\; ...,\; x+P_{m+k-1}(y)$, which has rather poor dependence on $m$ and degrees of $P_m, ..., P_{m+k-1}$. Second, the method does not generalize to higher $m$ without resorting to higher order Fourier analysis. For this reason, we do not present this argument here, hoping to find a more robust version of it in the future.


\begin{thebibliography}{alpha}

\bibitem[BC17]{bourgain_chang_2017}
J.~Bourgain and M.-C. Chang.
\newblock {Nonlinear Roth type theorems in finite fields}.
\newblock {\em Israel J. Math.}, \textbf{221}:853--867, 2017.

\bibitem[Beh46]{behrend_1946}
F.~Behrend.
\newblock On sets of integers which contain no three terms in arithmetical
  progression.
\newblock {\em Proc. Natl. Acad. Sci. USA}, \textbf{32}:331--2, 1946.

\bibitem[BL96]{bergelson_leibman_1996}
V.~Bergelson and A.~Leibman.
\newblock {Polynomial extensions of van der Waerden's and Szemer\'{e}di's
  theorems}.
\newblock {\em J. Amer. Math. Soc.}, \textbf{9}:725--753, 1996.

\bibitem[Blo16]{bloom_2016}
T.~Bloom.
\newblock {A quantitative improvement for Roth's theorem on arithmetic
  progressions}.
\newblock {\em J. Lond. Math. Soc.}, \textbf{93}:643--663, 2016.

\bibitem[BPPS94]{balog_pelikan_pintz_szemeredi_1994}
A.~Balog, J.~Pelik\'{a}n, J.~Pintz, and E.~Szemer\'{e}di.
\newblock {Difference sets without $k$th powers}.
\newblock {\em Acta Math. Hungar.}, \textbf{65}(2):165--187, 1994.

\bibitem[DLS17]{dong_li_sawin_2017}
D.~Dong, X.~Li, and W.~Sawin.
\newblock {Improved estimates for polynomial Roth type theorems in finite
  fields}.
\newblock 2017.

\bibitem[EG16]{ellenberg_gijswijt_2016}
J.~Ellenberg and D.~Gijswijt.
\newblock {On large subsets of $\FF_q^n$ with no three-term arithmetic
  progression}.
\newblock {\em Ann. of Math.}, \textbf{185}:339--343, 2016.

\bibitem[FKM13]{fouvry_kowalski_michel_2013}
E.~Fouvry, E.~Kowalski, and P.~Michel.
\newblock{An inverse theorem for Gowers norms of trace functions over $\FF_p$}.
\newblock{\em Math. Proc. Cambridge Philos. Soc.}, \textbf{155}:277--295, 2013.

\bibitem[Gow01]{gowers_2001}
W.~T. Gowers.
\newblock {A new proof of Szemer{\'e}di's theorem}.
\newblock {\em Geom. Funct. Anal.}, \textbf{11}(3):465--588, 2001.

\bibitem[Gow10]{gowers_2010}
W.~T. Gowers.
\newblock {Decompositions, approximate structure, transference, and the
  Hahn-Banach theorem}.
\newblock {\em Bull. Lond. Math. Soc}, \textbf{42}(4):573--606, 2010.

\bibitem[GR90]{graham_ringrose_1990}
S.~W. Graham and C.~J. Ringrose.
\newblock {Lower bounds for least quadratic non-residues}.
\newblock {\em Analytic Number Theory. Progress in Mathematics},
  \textbf{85}:269--309, 1990.

\bibitem[Gre07]{green_2007}
B.~Green.
\newblock {Montreal lecture notes on quadratic Fourier analysis}.
\newblock 2007.

\bibitem[GT10]{green_tao_2010}
B.~Green and T.~Tao.
\newblock Linear equations in primes.
\newblock {\em Ann. of Math.}, \textbf{171}:1753--1850, 2010.

\bibitem[GT17]{green_tao_2017}
B.~Green and T.~Tao.
\newblock {New bounds for Szemer\'{e}di's theorem, III: a polylogarithmic bound
  for $r_{4}(N)$}.
\newblock {\em Mathematika}, \textbf{63}(3):944--1040, 2017.

\bibitem[GW11a]{gowers_wolf_2011c}
W.~T. Gowers and J.~Wolf.
\newblock {Linear forms and higher-degree uniformity for functions on
  ${\mathbb{F}^{n}_{p}}$}.
\newblock {\em Geom. Funct. Anal.}, \textbf{21}:36--69, 2011.

\bibitem[GW11b]{gowers_wolf_2011b}
W.~T. Gowers and J.~Wolf.
\newblock {Linear forms and quadratic uniformity for functions on
  $\mathbb{F}_p^n$}.
\newblock {\em Mathematika}, \textbf{57}:215--237, 2011.

\bibitem[GW11c]{gowers_wolf_2011a}
W.~T. Gowers and J.~Wolf.
\newblock {Linear forms and quadratic uniformity for functions on
  $\mathbb{Z}_N$}.
\newblock {\em J. Anal. Math.}, \textbf{115}(1):121--186, 2011.

\bibitem[IK04]{iwaniec_kowalski_2004}
H.~Iwaniec and E.~Kowalski.
\newblock {\em Analytic Number Theory}.
\newblock AMS, 2004.

\bibitem[Luc06]{lucier_2006}
J.~Lucier.
\newblock Intersective sets given by a polynomial.
\newblock {\em Acta Arith.}, \textbf{123}:57--95, 2006.

\bibitem[Mes95]{meshulam_1995}
R.~Meshulam.
\newblock On subsets of finite abelian groups with no 3-term arithmetic progressions.
\newblock {\em J. Comb. Theory Ser. A}, \textbf{71}(1):168--172, 1995.

\bibitem[Pel18]{peluse_2018}
S.~Peluse.
\newblock Three-term polynomial progressions in subsets of finite fields.
\newblock {\em Israel J. Math.}, \textbf{228}:379--405, 2018.

\bibitem[Pel19a]{peluse_2019a}
S.~Peluse.
\newblock {Bounds for sets with no polynomial progressions}.
\newblock 2019.

\bibitem[Pel19b]{peluse_2019b}
S.~Peluse.
\newblock {On the polynomial Szemer\'{e}di theorem in finite fields}.
\newblock {\em Duke Math. J.}, 2019.

\bibitem[PP19]{peluse_prendiville_2019}
S.~Peluse and S.~Prendiville.
\newblock {Quantitative bounds in the non-linear Roth theorem}.
\newblock 2019.

\bibitem[Pre17]{prendiville_2017}
S.~Prendiville.
\newblock {Quantitative bounds in the polynomial Szemer{\'e}di theorem: the
  homogeneous case}.
\newblock {\em Discrete Anal.}, \textbf{5}, 2017.

\bibitem[Ric19]{rice_2019}
A.~Rice.
\newblock {A maximal extension of the best-known bounds for the
  Furstenberg-S\'ark\"ozy theorem}.
\newblock {\em Acta Arith.}, \textbf{187}:1--41, 2019.

\bibitem[Ruz84]{ruzsa_1984}
I.~Ruzsa.
\newblock Difference sets without squares.
\newblock {\em Period. Math. Hungar.}, \textbf{15}(3):205--209, 1984.

\bibitem[RW19]{rimanic_wolf_2019}
L.~Rimani\'{c} and J.~Wolf.
\newblock {Szemer\'{e}di's theorem in the primes}.
\newblock {\em Proc. Edinb. Math. Soc.}, \textbf{62}:443--457, 2019.

\bibitem[S{\'a}78a]{sarkozy_1978a}
A.~{S\'{a}rk\"{o}zy}.
\newblock {On difference sets of sequences of integers. I}.
\newblock {\em Acta Math. Hungar.}, \textbf{31}(1-2):125--149, 1978.

\bibitem[S{\'a}78b]{sarkozy_1978b}
A.~{S\'{a}rk\"{o}zy}.
\newblock {On difference sets of sequences of integers. III}.
\newblock {\em Acta Math. Hungar.}, \textbf{31}:355--386, 1978.

\bibitem[Sli03]{slijepcevic_2003}
S.~Slijep\u{c}evi\'{c}.
\newblock {A polynomial S\'{a}rk\"{o}zy-Furstenberg theorem with upper bounds}.
\newblock {\em Acta Math. Hungar.}, \textbf{98}(1-2):111--128, 2003.

\bibitem[Sze75]{szemeredi_1975}
E.~Szemer\'{e}di.
\newblock On sets of integers containing $k$ elements in arithmetic
  progression.
\newblock {\em Acta Arith.}, \textbf{27}(1):199--245, 1975.

\bibitem[Tao12]{tao_2012}
T.~Tao.
\newblock {\em {Higher order Fourier analysis}}.
\newblock AMS, 2012.

\bibitem[TV06]{tao_vu_2006}
T.~Tao and V.~Vu.
\newblock {\em Additive Combinatorics}.
\newblock Cambridge Studies in Advanced Mathematics. Cambridge U. P., 2006.

\bibitem[TZ08]{tao_ziegler_2008}
T.~Tao and T.~Ziegler.
\newblock{The primes contain arbitrarily long polynomial progressions}.
\newblock{\em Acta Math.}, \textbf{201}(2): 213--305, 2008.

\bibitem[TZ16]{tao_ziegler_2016}
T.~Tao and T.~Ziegler.
\newblock{Concatenation theorems for anti-Gowers-uniform functions and Host–
Kra characteristic factors}.
\newblock{\em Discrete Anal.}, \textbf{13}: 61 pp, 2016.

\bibitem[TZ18]{tao_ziegler_2018}
T.~Tao and T.~Ziegler.
\newblock{Polynomial patterns in the primes}.
\newblock{\em Forum Math. Pi}, \textbf{6}, 2018.


\end{thebibliography}

\end{document}